\documentclass[a4paper,10pt]{article}
\usepackage{amsmath,amssymb, amsthm, graphicx, a4wide, amsopn}
\usepackage{arydshln}
\usepackage{enumerate, framed}
\usepackage[explicit]{titlesec}
\usepackage[dvipsnames]{xcolor}
\usepackage{nicefrac}

\newif\ifdraft		\draftfalse
\newif\ifshowkeys	\showkeysfalse
\newif\ifswap		\swaptrue	

\allowdisplaybreaks

\usepackage{tikz}
\usetikzlibrary{matrix,arrows,decorations.pathmorphing}
\usepackage{tikz-cd}

\usetikzlibrary{arrows}

\tikzset{
    arrow/.style = {-stealth},
}

\usepackage{typearea}

\areaset[current]{376pt}{680pt} 
\newcommand{\gyzFormat}[1]{#1}

\ifswap
\newtheoremstyle{mythmstyle}
  {\topsep}
  {\topsep}
  {\normalfont}
  {}
  {}
  {{\bf .}}
  {.7em}
  {{{\bfseries\thmnumber{#2}}~\gyzFormat{\bfseries\thmname{#1}}{\normalfont\gyzFormat{\thmnote{ (#3)}}}}}
\else
\newtheoremstyle{mythmstyle}
  {\topsep}
  {\topsep}
  {\normalfont}
  {}
  {}
  {\bfseries.}
  {.7em}
 {{\gyzFormat{\bfseries\thmname{#1}}~{\bfseries\thmnumber{#2}}{\normalfont\gyzFormat{\thmnote{ (#3)}}}}}
\fi

\theoremstyle{mythmstyle}  
\newtheorem{theorem}{Theorem}[section]
\newtheorem{definition}[theorem]{Definition}
\newtheorem{lemma}[theorem]{Lemma}

\newtheorem{remark}[theorem]{Remark}

\newtheorem{proposition}[theorem]{Proposition}
\newtheorem{corollary}[theorem]{Corollary}
\newtheorem{claim}[theorem]{Claim}

\newtheorem{example}[theorem]{Example}

\newtheoremstyle{mynotestyle}
  {\topsep}
  {\topsep}
  {\normalfont}
  {}
  {}
  {\bfseries.}
  {.7em}
  {{{\bfseries\usefont{T1}{yvtj}{b}{n}\thmnumber{#2}}{\normalfont\gyzFormat{\thmnote{ (#3)}}}}}

\theoremstyle{mynotestyle}  

\renewenvironment{proof}[1][\unskip]{%
\par
\noindent
\textbf{Proof #1.}
\noindent}
{\hfill$\blacksquare$

\medskip}

\let\<=\langle
\let\>=\rangle

\def\land{\wedge}       \def\lor{\vee}
\def\Land{\bigwedge}    
\def\lnot{\neg}

\def\[#1\]{\begin{align}#1\end{align}}

\def\endef{\hfill$\square$}

\DeclareMathOperator{\Rng}{ran}

\DeclareMathOperator{\Dom}{dom}

\DeclareMathOperator{\dom}{dom}
\DeclareMathOperator{\Cg}{Cg}

\def\vd{\vdash}
\def\nvd{\nvdash}

\let\phi=\varphi
\let\theta=\vartheta
\let\epsilon=\varepsilon

\let\models=\vDash

\renewcommand{\setminus}{\smallsetminus}

\def\gA{\mathfrak{A}}
\def\gB{\mathfrak{B}}
\def\gC{\mathfrak{C}}

\def\gM{\mathfrak{M}}
\def\gN{\mathfrak{N}}

\def\cL{\mathcal{L}}

\def\TITLE{Logic families}
\def\AUTHOR{Hajnal Andr\'eka\thanks{Alfr\'ed R\'enyi Institute of Mathematics, Budapest}\and Zal\'an Gyenis\thanks{Jagiellonian University, Krak\'ow}\and Istv\'an N\'emeti${}^{*}$\and Ildik\'o Sain${}^{*}$}
\def\DATE{\today}
\def\ABSTRACT{A logic family is a bunch of logics that belong together in some way. 
		First-order logic is one of the examples. Logics organized into a structure occur in abstract model theory, institution theory and in algebraic logic. Logic families play a role in adopting methods for investigating sentential logics to first-order like logics. We thoroughly discuss the notion of logic families as defined in the recent Universal Algebraic Logic book.
	 } 
\def\KEYWORDS{Logics, general logics, universal algebraic logic, abstract algebraic logic, abstract model theory, institutions, Lindenbaum--Tarski formula algebra, tautological congruence, conditionally free algebra, free product of algebras}

\def\LL{\mathbf{L}}
\def\F{\mathfrak{F}}

\def\SP{\mathbf{SP}}
\def\HSP{\mathbf{HSP}}
\def\ACSA{\Delta}				
\DeclareMathOperator{\Alg}{Alg}
\DeclareMathOperator{\Sig}{Sig}

\DeclareMathOperator{\mng}{mng}
\DeclareMathOperator{\Hom}{Hom}
\DeclareMathOperator{\Cn}{Cn}

\def\Fr{\mathfrak{Fr}}

\def\K{\mathsf{K}}
\def\M{\mathsf{M}}
\def\CPL{\mbox{$\mathcal C\mathcal P\mathcal L$}}
\def\FOL{\mbox{$\mathcal F\mathcal O\mathcal L$}}

\def\Mng{\mbox{Mng}}

\def\Id{\mbox{\sf Id}}

\def\si{\mbox{\sf si}}
\def\Cr{\mbox{\sf Cr}}

\def\Trm{\mbox{\sf Trm}}

\def\TT{\mathbf{T}}

\begin{document}
	
 \title{{\bf \TITLE}}
 \author{\AUTHOR} \date{\DATE}
 \maketitle
 \thispagestyle{empty}
	
 \begin{abstract}
 \ABSTRACT
 \vspace{5mm}

 \end{abstract}
 \vspace{5mm} \normalsize

\section{Introduction}\label{sec:intro}

This paper is about a notion of a family of logics. A family of logics is a bunch of logics that belong together in some way. A typical example is first-order logic. One defines first-order logic by first specifying the similarity type and after this one defines first-order logic of that similarity type. Thus, the similarity type is  a parameter in defining first-order logic. The present paper contains a thorough classical style discussion of a notion of logic families which focuses on the set of atomic formulas being a parameter.

Several notions for logic families occur in the logic literature, e.g., \cite{Ba74}, \cite{Diabook}, \cite{HooSL00}, \cite{AGyNS2022}. In all these, one relies on a general notion of logics, and then a logic family is a class of such logics organized into some structure.
Before talking about the motivation for logic families, let us begin with saying some words about their constituents. 
Below we list some situations which call for a general notion of a logic.

----- Often, one says that one property of a logic follows from another one. For example, it is widely known that there are close connections between Craig’s interpolation theorem, Robinson’s consistency theorem and Beth's definability theorem in first-order logic (e.g., one can prove one theorem from the other). However, it is just three 
facts of mathematics that these properties hold for first-order logic, one needs to give meaning to statements concerning their relationship. 

----- The starting point of algebraic logic is the discovery that in many cases, a natural logical property corresponds in some way to a natural algebraic property \cite{Pigo72}. 
For making such correspondences meaningful, one needs a general notion of logics. 

----- The main aim of abstract model theory was to find a useful extension of first-order logic: one which is stronger while manageable and expressive \cite{Ba74}. For such an endeavour one needs a general notion of logics among which one wants to select a good one.

----- In computing theory, especially in formal specification theory, there was an explosion of formal logical systems used, and so there was a need for a uniform treatment of specification concepts. Institution Theory was created to fulfill this need, but it reaches much further, it is an axiomatic treatment of classical and non-classical model theories.%
\footnote{See Goguen-Burstall \cite{GB92}, Diaconescu \cite{Diabook}, \cite{Dia2022}.}

As seen from the above examples, even those logicians  for whom first-order logic is the one and only may use a general notion of logics in order to make ideas about first-order logic meaningful. Let us return to logic families.

In Abstract Model Theory, families of logics are used for doing away with variables.%
\footnote{See \cite{Ba74}. Papers using the framework of Abstract Model Theory are, e.g., \cite{Mu81}, \cite{LaMa15}.}
Institution Theory 
relies on a general notion of logics $\langle F,M,\models\rangle$ where $F$ is the set of formulas, $M$ is the class of models and $\models$ is the validity relation between them. An institution is a class of such logics organized into a category $\Sig$. Thus, the notion of logic family is central in Institution Theory.%
\footnote{See \cite{Diabook}, \cite{SaTa12}, \cite{Vou}.}  Abstract Algebraic Logic and Universal Algebraic Logic are two branches of Algebraic Logic that deal with general notions of logic. Both grew out of, and was influenced by, Alfred Tarski's work. The first one concentrates on sentential logics  while the second concentrates more on predicate-like logics. However, their methods are very similar to each other.%
\footnote{See Blok-Pigozzi \cite{BP}, Font \cite{Font}, Font-Jansana \cite{FJ94c}, and \cite{AGyNS2022}.} 
One benefit of logic families in these branches is aesthetic: the cardinality of the set of formulas often shows up in theorems about logics, but this cardinality has no significant meaning. 
By using logic families instead of logics, the corresponding theorems can do away with these cardinalities. Besides this, there are theorems in which the use of logic families is essential. For example, \cite[Theorem 4.3.44]{AGyNS2022} says that a logic family has the weak Beth definability property if and only if in its counterpart category the class of maximal elements generates the whole category by limits. In the counterpart category, all members of the logic family are taken into account and being maximal here cannot be expressed by using single logics instead of the whole family.

The recent book Universal Algebraic Logic \cite{AGyNS2022} aims at presenting the most important theorems and methods of Universal Algebraic Logic. The key feature of logics treated in Abstract Algebraic Logic is that they are substitutional%
\footnote{The word structural is used in Abstract Algebraic Logic, see, e.g., \cite{BP}, \cite{CD06}, \cite{Font}.}
 which means that the atomic formulas do not carry any other information than how many of them are there. First-order logic, clearly, is not substitutional. 
 Yet, most of the methods of Abstract Algebraic Logic can be applied to first-order logic, after some modifications. Two moves are crucial in these modifications. One move is generalizing the notion of being substitutional to being conditionally substitutional, the other move is refining the notion of a logic family%
 \footnote{This notion is called ``general logic'' in \cite{AGyNS2022}, \cite{AKNS}, Font-Jansana \cite{FJ94c}, Hoogland \cite{HooSL00}.} that is used in Abstract Algebraic Logic. Thanks to these two moves, the algebraic methodology á la \cite{AGyNS2022} can be directly applied to first-order logic. Thus the definition of a logic family in [1] is important in moving from sentential logics to predicate-like ones. It may deserve a thorough investigation.

The content of this paper is an addition to \cite{AGyNS2022} while it is self-contained. We expand and explain some features of the definition of logic families.  We carry on a classical style discussion of \cite[Definition 3.3.27]{AGyNS2022}: we simplify it, we present several equivalent statements for the items occurring in it, we present weaker and stronger statements for these items; we show that under some hypotheses, some of the distinct statements become equivalent etc. We hope that these investigations give insight concerning the definition of a logic family. 

The structure of the paper is as follows.

In section \ref{sec:logics}, we briefly introduce our formal notion of a logic alongside with some of its basic properties we will need 
in the paper. We also introduce classical propositional and first-order logics in this framework.

In section \ref{sec:subs}, we define two characteristic subclasses of logics. In the first class, atomic formulas can be replaced (substituted) with any other formulas without affecting meaning, in the second class, this can be done only with formulas satisfying some conditions. Propositional logics usually fall into the first class, while classical first-order logic is a typical example of the second class.

Section \ref{sec:families} is the main part of the paper, it contains a detailed discussion of the notion of logic families. 
Subsection \ref{GLdef-subsec} recalls the definition of a logic family from \cite{AGyNS2022}, with slight modifications. This definition contains five items (1)--(5). Of these, condition (4) is about union of signatures, it is the one that needs most explanation. 
Subsection \ref{GLdef-subsec} contains a discussion of items (1)-(3) and (5).
Subsection \ref{disc4-subsec} discusses condition (4) in the general case. For showing the intuition behind (4), we present three consequently weaker conditions (4a), (4b) and (4). Of these, (4a) is practically useless because it almost never is satisfied (though, it reflects the intuition best), while condition (4b) is equivalent to the condition used in \cite[Def.3.3.27]{AGyNS2022}. Theorem \ref{thm:4b} lists six conditions of different characters all of which are equivalent to (4b).  Theorem \ref{gen4} lists seven conditions that are equivalent to (4) and analogous to the ones that characterize (4b). These theorems shed light on why our (4) is weaker than (4b) of \cite{AGyNS2022}. 
Subsection \ref{subs:cs4} investigates characterizations that hold only for conditionally substitutional logic families. These give a hint for why condition (4) may be more useful in the definition than (4b). At the end of the subsection it is shown that the notion of logic families is not restrictive among conditionally substitutional logics: each conditionally substitutional logic is a member of some conditionally substitutional logic family, while there is a logic that is member of no logic family at all (Theorem \ref{thm:restr}).

Section \ref{sec:exa} is also an important part of the paper. In this section we prove that both classical propositional and first-order logics are members of natural logic families. It may come as a surprise that the systems of finite-variable fragments of first-order logic as defined in this subsection do not satisfy condition (4) (Theorem \ref{thm:nofamily}).

\bigskip

\section{Logics in a general setting}\label{sec:logics}
In this section, we briefly introduce the formal notion of a logic alongside with some of its basic properties we will need 
in the paper. We also introduce classical propositional and first-order logics in this framework.

\subsection{Definition of a logic and two examples}\label{sec:logic}
In this subsection, we recall a general notion of logic from \cite{AGyNS2022} and we cast classical propositional and first-order logics as examples.
Similar but different definitions of a logic 
(or, formal system) can be found in, e.g., Blok-Pigozzi \cite{BP}, Henkin-Monk-Tarski \cite[section 5.6]{HMT}, Tarski-Givant \cite[section 1.6]{TG}. 
Reasons for defining such a notion can be found in the introduction. 
What we define below is called in \cite{AGyNS2022} a compositional logic having connectives, see \cite[Defs. 3.1.3, 3.3.1, 3.3.2]{AGyNS2022}. There are some differences though, we mention these in the discussion. For the intuition behind the definition of a logic below we refer to \cite[Sec. 3.1]{AGyNS2022} and also to the two examples that follow the definition.

\begin{definition}[logic]\label{def:logic}
	By a logic we understand a tuple $\cL = \< F, M, \mng, \models\>$, 
	where (i)--(iv) below hold.
	\begin{enumerate}[(i)]\itemsep-2pt
		\item  $F$, called the set of formulas, is
		the universe of a term-algebra%
		\footnote{For the definition of term-algebra see, e.g., \cite[Def. 2.4.1]{AGyNS2022}}
		$\F(P, \Cn)$, which is an absolutely free algebra over $P$ in the type of $\Cn$.	
		The domain of the type $\Cn$ is called the set of logical connectives 
		and $P$ is called the set of atomic formulas. 
		We also denote the set of logical connectives with $\Cn$, and the algebra
		$\F$ is many times referred to as the formula algebra of $\cL$.
		\item $M$ is a class, called the class of models.
		\item $\models\;\subseteq M\times F$ is a relation between models and formulas. It is called the validity relation.
		We use the notation $\gM\models\phi$ instead of $\<\gM, \phi\>\in\;\models$.
		\item $\mng$ is called the meaning function. It is a function with domain $M\times F$ (we write $\mng_{\gM}(\phi)$ in 
		place of $\mng(\gM,\phi)$) such that
		\begin{itemize}
			\item $(\forall \phi, \psi\in F)\; (\forall \gM\in M)\;
	\big(\mng_{\gM}(\phi) = \mng_{\gM}(\psi) \text{ and } \gM\models\phi \big)
	\Longrightarrow \gM\models\psi$.
			\item We assume \emph{compositionality}, that is, $\mng_{\gM}$ is a homomorphism 
			from the formula algebra $\F$ to some algebra of type $\Cn$.
		\end{itemize}
	\end{enumerate}
	\endef
\end{definition}

Many well-known logics can be cast in our framework, \cite{AGyNS2022} contains numerous examples.
Throughout this paper we illustrate our definitions with two examples only: classical propositional logic and first-order logic.

\begin{example}[Classical propositional logic \CPL]\label{def:proplogic}
	We define classical propositional logic \CPL\ as follows.	
	\begin{itemize}
		\item The set of logical connectives is $\{\land, \lnot,\bot\}$, where $\land$ is binary, $\lnot$ is unary
		and $\bot$ is a constant; and $P$ is an arbitrary set (of atomic propositions) containing no proper term written 
		up with these connectives\footnote{For the name proper term see the next subsection.}. 
		Then the set $F_C$ of formulas is the universe of the corresponding term-algebra.
		\item The models, elements of $M_C$, are functions $\gM:P\to\{0,1\}$ assigning $0$ (false) and $1$ (true) to each atomic proposition $p\in P$. 
	\item We extend any model $\gM:P\to \{0,1\}$ to the set of all formulas by
	\begin{equation}
		\gM(\bot) = 0,\quad \gM(\lnot\phi) = 1-\gM(\phi),\quad \gM(\phi\land\psi)=
		\gM(\phi)\cdot\gM(\psi).
	\end{equation}
	We let the meaning of a formula $\phi$ in a model $\gM$ be $\gM(\phi)$, i.e., $\mng_{C}(\gM,\phi)=\gM(\phi)$.
	\item A formula is valid in a model iff its meaning is ``true'', i.e., $\gM\models_C\phi\Longleftrightarrow\gM(\phi)=1$.
\end{itemize}
Thus classical propositional logic \CPL\ in our formalism is $$\CPL=\langle F_C, M_C, \mng_C, \models_C\rangle.$$
\end{example}

\bigskip
\noindent Next we turn to first-order logic, one of the key examples of our framework.
\bigskip

\begin{example}[First-order logic $\FOL = \FOL_t(V)$]\label{def:FOL}
	We formalize first-order logic of a non-empty relational 
	similarity type $t$ using variables $V$ in
	the standard way, as follows. The set $P_t(V)$ of atomic formulas is
	\begin{equation}
		P_t(V) = \{ r(v_{1},\ldots,v_{{n}}) :\; 
		r\in\dom(t),\; t(r)=n,\; v_{1},\dots,v_n\in V\}.\label{eq:FOLPtV}
	\end{equation}
	The set $F_t(V)$ of formulas is generated by $P_t(V)$ using the connectives 
	$\{\land,\lnot,\bot,\exists v, v=w : v,w\in V\}$, 
	where $\land$ is binary, $\lnot, \exists v$ are unary, and 
	$v=w$ and $\bot$ are constants. 
	Models $\gM=\< M,r^{\gM}\>_{r\in\dom(t)}$
	and satisfaction $\models^t$ are defined in the usual manner see, e.g., Chang-Keisler \cite[section 1.3]{CK}, or Enderton \cite[sections 2.0--2.3]{End}. Finally, 
	\begin{align}
		\mng_{\gM}(\phi)=\{ h\in{}^{V}M : \gM\models\phi[h]\},
	\end{align}
	where ${}^{V}M$ denotes the set of all functions from $V$ to $M$, and $\gM\models\phi[h]$ denotes that $\phi$ is true in $\gM$ under the evaluation $h$ of the variables.	We define 
	\begin{align}
		\FOL_t(V) = \< F_t(V), M_t, \mng^t, \models^t \>.
	\end{align}
	We note that the pair $\<t,V\>$ and the set $P_t(V)$ of atomic formulas 
	determine each other. 
\emph{Classical first order logic} is $\FOL_t(V)$ where the set $V$ of variables is infinite. When $V$ is finite, $\FOL_t(V)$ is called a \emph{finite-variable fragment of first-order logic}. We will see that these finite-variable fragments have properties different from those of classical first-order logic and so they will serve as insightful counterexamples to statements about the latter.	
\end{example}

\subsection{Some basic properties of logics}\label{ssec:basics}
In this subsection, we briefly recall some basic definitions connected to logics and some basic results that we will use in the paper.
Let $\cL$ be a logic in the sense of Definition \ref{def:logic}. The sets of connectives, atomic formulas, formulas, models etc of $\cL$ will be denoted by $\Cn(\cL), P(\cL), F(\cL), M(\cL)$ etc. Sometimes we write $\Cn_{\cL}, P_{\cL}$ etc.\ or we omit reference to $\cL$ when it is clear from context.

Recall from Definition \ref{def:logic} that the \emph{formula algebra} $\F$ of $\cL$ is the term-algebra over $P$ of similarity type $\Cn$. This means that the universe of $\F$ is the set of terms constructed by the logical connectives from the atomic formulas, e.g., $\lnot p$ is a term if $\lnot$ is a unary connective and $p$ is an atomic formula. The operations of the term-algebra are the natural ones imitating this construction, e.g., there is a unary operation in $\F$ denoted by $\lnot$. Note that the operations of the formula algebra are denoted by the connectives. 

In the definition of a logic we require that $\F$ is an absolutely free algebra over $P$. This means that any function $f:P\to A$ where $A$ is the universe of an algebra $\gA$ of similarity type $\Cn$ can be extended to a homomorphism $\bar{f}:\F\to\gA$. However,  the term-algebra
$\F$ is an absolutely free algebra over $P$ if and only if no element of $P$ is a ``proper'' term written up by $\Cn$ from other elements of $P$, i.e., if $p\notin\Trm_{\Cn}(P\setminus\{ p\})$, for all $p\in P$. We will refer to this property briefly as $P$ does not contain proper terms by $\Cn$. Similarly, we need that $P$ and $\Cn$ are determined by $\F$, this holds when $\F$ is an absolutely free algebra over $P$. This is why the condition that no element of $P$ is a proper term by $\Cn$ occurs at some places in the paper.

Define a congruence relation $\sim_{\cL}$ on the formula algebra $\F$ by
\begin{equation}
	(\forall \phi,\psi\in F)\big( \phi\sim_{\cL}\psi \Longleftrightarrow
	(\forall \gM\in M) \mng_{\gM}(\phi) = \mng_{\gM}(\psi) \big),
\end{equation}
that is, $\sim_{\cL} \;=\; \bigcap\{ \ker(\mng_{\gM}):\; \gM\in M\}$,
where $\ker(f)$ denotes the \emph{kernel} of the function $f:A\to B$, i.e., $\ker(f) = \{ \< a, b\> : a, b\in A, f(a)=f(b)\}$. 
$\sim_{\cL}$ is called the \emph{tautological congruence} relation of 
$\cL$. It is a congruence relation because all the meaning functions are homomorphisms.
Two formulas are related by the tautological congruence iff their meanings coincide in all models, i.e., if no model can distinguish them based on their meanings.

The quotient algebra $\F/\!\!\sim_{\cL}$ of $\F$ factorized
by $\sim_{\cL}$ is called the 
\emph{tautological formula algebra}, or the \emph{Lindenbaum--Tarski algebra
	of $\cL$}. 
In many logics, e.g., in $\CPL$ and $\FOL$, the tautological congruence relation corresponds to the set of formulas valid in all models. In more detail, 
coincidence of meanings of formulas and the validity relation are interdefinable in these logics. For example, $\mng_{\gM}(\phi)=\mng_{\gM}(\psi)$ holds iff $\gM\models \phi\leftrightarrow\psi$ in $\CPL$.%
\footnote{In general logic, this property of a logic is called the filter property see, e.g., \cite[Def. 3.3.18]{AGyNS2022}}.

In any logic, there are certain properties that hold for all formulas in place of the atomic ones. For example in $\FOL$, the formula $\forall x[\exists yR(x)\to R(x)]$ does not hold for all formulas in place of $R(x)$, but $R(x)\to R(x)$ does. This motivates the following definition.

A function $s:P\to F$ is called a formula-substitution, or just a \emph{substitution}. Since $s$ extends uniquely to a homomorphism $\bar{s}:\F\to\F$, all homomorphisms $h:\F\to\F$ are also called substitutions. The \emph{substitution-invariant part} of $\sim_{\cL}$ is defined as
\begin{equation}\label{si}
	\si(\sim_{\cL}) \ =\  \{\<\phi,\psi\> : (\forall h:\F\to\F)\ \< h(\phi),h(\psi)\>\in\ \sim_{\cL}\} .
\end{equation}

For a model $\gM\in M$, the algebra $\mng_{\gM}(\F)$ is called the \emph{concept algebra} or the\emph{ meaning algebra} of $\gM$. The name concept algebra reflects that the meanings of formulas in $\gM$ can be considered to be the concepts of $\gM$ expressible in $\cL$. Let $\Mng(\cL)$ and $\Alg_m(\cL)$ denote the classes of all meaning functions and all meaning algebras, respectively:
\begin{equation}\label{mngalg}
	\Mng(\cL)=\{\mng_{\gM} : \gM\in M\}\quad\mbox{ and }\quad\Alg_m(\cL) = \big\{\mng_{\gM}(\F):\; \gM\in M \big\}.
\end{equation}
Both $\Mng(\cL)$ and $\Alg_m(\cL)$ are counterparts of models. Usually, but not always,  $\mng_{\gM}$ has all the information about the model, i.e., if $\gM,\gN$ are distinct, then $\mng_{\gM}$ and $\mng_{\gN}$ are also distinct. A concept algebra retains all information the meaning function has, except that 
the information about which formula is atomic is ``forgotten'' (deliberately). 

Theorem \ref{fmalg} below connects the notions introduced so far. In the theorem, just as in the paper, two notions related to that of the term-algebra will be extensively used: the notions of free algebras and conditionally free algebras. In the following, we omit reference to the similarity type since this can be determined by the algebra. Let $\F(X)$ denote the term-algebra of some similarity type over a set $X$, let $\K$ be a class of algebras of the same similarity type and let $S\subseteq F\times F$ be a binary relation over the universe $F$ of $\F(X)$. Let $\Hom(\F,\K)$ denote the class of all homomorphisms from $\F$ to members of $\K$, and let 
\[\Hom(\F,\K,S) = \{ h\in\Hom(\F,\K) : S\subseteq\ker(h)\} .\]
The \emph{$X$-freely generated $\K$-free algebra} $\Fr(\K,X)$ is defined  as
\[ \Fr(\K,X)=\F(X)\slash\Cr(\K,X),\text{ where } \Cr(\K,X)=\bigcap\{\ker(h) : h\in\Hom(\F(X),\K)\} .\]
Similarly, let
\[\Cr(\K,X,S)=\bigcap\{\ker(h) : h\in\Hom(\F(X),\K,S)\}\]
and the \emph{$X$-freely generated conditionally $\K$-free algebra with conditions  $S$} is defined as
\[\Fr(\K,X,S) =  \F(X)\slash\Cr(\K,X,S) .\]
For a class $\K$ of similar algebras, let $\SP\K$ and $\HSP\K$ denote the classes of algebras obtained from elements of $\K$ by the algebraic constructions of products and subalgebras, and the algebraic constructions of products, subalgebras and homomorphisms, respectively. The term-algebra and the free algebra are special cases of the conditionally free algebra, namely $\Fr(\K,X)=\Fr(\K,X,\emptyset)$ and $\F(X)=\Fr(\M,X)$ where $\M$ is the class of all algebras of the similarity type of $\F(X)$. We will use that $\Fr(\K,X)=\Fr(\HSP\K,X)$ and $\Fr(\K,X,S)=\Fr(\SP\K,X,S)$. For more about the conditionally free algebras see \cite[0.4.65]{HMT}.

Theorem \ref{fmalg} below states that the tautological formula algebra of any logic is the $P$-freely generated $\sim_{\cL}$-conditionally $\Alg_m(\cL)$-free algebra,%
\footnote{In \cite[Cor. 3.3.16]{AGyNS2022} this is stated only for conditionally substitutional logics, but the statement holds for all logics.}
and the formula algebra factored by the substitution-invariant part of $\sim_{\cL}$ is the $P$-freely generated $\Alg_m(\cL)$-free algebra. For $\phi,\psi\in F$, let $[\phi=\psi]$ denote the equation in the language of $\Alg_m(\cL)$ where the elements of $P$ are the algebraic variables.  Theorem \ref{fmalg}\eqref{eq} states that the substitution-invariant part of $\sim_{\cL}$ is the equational theory of $\Alg_m(\cL)$ such that the atomic formulas are the algebraic variables.
Thus, the equational theory of the concept algebras can be recovered from the tautological congruence by syntactical means, in any logic.

\begin{theorem}\label{fmalg} Let $\cL$ be any logic in the sense of Definition \ref{def:logic}. Then (i)-(iii) below hold.
\begin{enumerate}[(i)]\itemsep-2pt
	\item $\F(\cL)\slash\!\sim_{\cL} = \Fr(\Alg_m(\cL),P(\cL),\sim_{\cL})$
	\item $\F(\cL)\slash\si(\sim_{\cL}) = \Fr(\Alg_m(\cL),P(\cL))$
	\item\label{eq}  $\si(\sim_{\cL}) = \{\<\phi,\psi\> : \Alg_m(\cL)\models[\phi=\psi]\}$
\end{enumerate}
\end{theorem}
\begin{proof} 
	(i) We show that $\sim_{\cL}$ equals $\Cr(\Alg_m(\cL), P(\cL), \sim_{\cL})$,
	that is,
	\begin{equation}\label{eq:simcr}
		\bigcap\{ \ker(\mng_{\gM}):\; \gM\in M(\cL) \}
		\ =\ 
		\bigcap\{ \ker(h):\; h\in\Hom(\F(\cL), \Alg_m(\cL), \sim_{\cL}) \}\,.
	\end{equation}
	For $h\in\Hom(\F(\cL), \Alg_m(\cL), \sim_{\cL})$, we have 
	$\sim_{\cL}\;\subseteq\ker(h)$, therefore $\sim_{\cL}$ is a subset
	of the intersection in the right-hand side of \eqref{eq:simcr}, i.e., 
	$\sim_{\cL}$ \ $\subseteq$ $\Cr(\Alg_m(\cL), P(\cL), \sim_{\cL})$.
	For the inclusion $\supseteq$ note that for each $\gM\in M(\cL)$
	we have $\sim_{\cL}\;\subseteq \ker(\mng_{\gM})$, hence
	\[
		\mng_{\gM}\ \in\  \Hom(\F(\cL), \Alg_m(\cL), \sim_{\cL}),
	\]
	and thus $\Cr(\Alg_m(\cL), P(\cL), \sim_{\cL})$\ $\subseteq$ $\sim_{\cL}$.\\
	
\noindent (ii) Here we need to prove $\si(\sim_{\cL})\ =\ \Cr(\Alg_m(\cL), P(\cL))$, 
	that is, 
	\begin{align}\begin{split}\label{eq:proofii}
		\{ \<\phi,\psi\>:\; \forall h\in\Hom(\F,\F)\; &(\forall \gM\in M)\; 
		\mng_{\gM}(h(\phi)) = \mng_{\gM}(h(\psi)) \}\\
		&=\ \  \bigcap\{\ker(k):\; k\in\Hom(\F,\Alg_m(\cL)) \}\,.
	\end{split}\end{align}
	Take any $k\in\Hom(\F,\Alg_m(\cL))$, that is, $k:\F\to\mng_{\gM}(\F)$ for some
	$\gM\in M$. To prove \eqref{eq:proofii} it is enough to show that
	there is a homomorphism $h:\F\to\F$ such that 
	$k = \mng_{\gM}\circ\; h$. To obtain such $h$, note that 
	for each $p\in P(\cL)$ there is $\psi_p\in F$ such that 
	$k(p)=\mng_{\gM}(\psi_p)$. Let
	$h$ be the homomorphic extension of the mapping $p\mapsto\psi_p$. 
	\begin{center}
		\begin{tikzcd}[column sep=large]
			\F \arrow{r}{k} \arrow[swap,dashed]{d}{h} & \mng_{\gM}(\F) \\
			 \F\arrow[swap]{ur}{\mng_{\gM}} & 
		\end{tikzcd} 
	\end{center}

	(iii) Let $\gA=\mng_{\gM}(\F)$ for some $\gM\in M$. By definition, 
	$\gA\models [\phi=\psi]$ holds iff for all evaluations (homomorphism)
	$e:\F\to\gA$ we have $\gA\models [\phi=\psi][e]$, that is, the equality holds in
	$\gA$ under the evaluation $e$. Part (ii) of this proof shows that 
	$e = \mng_{\gM}\circ\; h$ for some homomorphism $h:\F\to\F$. Conversely, 
	for each $h:\F\to\F$ the homomorphism $\mng_{\gM}\circ\; h$ gives
	rise to an evaluation over $\gA$. Therefore,
    \begin{align}
        \gA\models[\phi=\psi] \ \  \text{ iff }\ \ 
        (\forall h:\F\to\F)\ \<h(\phi), h(\psi)\>\in\ker(\mng_{\gM}). \label{eqeqeq}
    \end{align}
To complete the proof
	recall that 
	\begin{align}
		\<\phi,\psi\>\in\si(\sim_{\cL}) \ \ \text{ iff }\ \ 
		(\forall \gM\in M)(\forall h:\F\to\F)\ \<\phi,\psi\>\in\ker(\mng_{\gM}\circ\;h)\,.
	\end{align}
\end{proof}
\bigskip

Recall the notion of a deductive system from \cite{BP}, \cite{Font}, \cite{BH06}. A deductive system $\vd$ is called \emph{2-dimensional} when $H\vdash\Phi$ implies $H\cup\{\Phi\}\subseteq F\times F$. We will deal only with 2-dimensional deductive systems in this paper, so we will omit the adjective 2-dimensional. When $\vdash$ is a deductive system and $\Sigma\subseteq F\times F$, we say that $\Sigma$ is \emph{closed} under $\vdash$ when $\<\tau,\sigma\>\in\Sigma$ whenever $H\vdash\<\tau,\sigma\>$ and $H\subseteq\Sigma$. The \emph{$\vdash$-closure} of $\Sigma$ is defined as the smallest set containing $\Sigma$ and closed under $\vdash$. We call a deductive system \emph{sound} if $\ker(\mng_{\gM})$ is closed under $\vdash$, for all models $\gM\in M$.  A deductive system is called \emph{substitutional} when $H\vd\<\tau,\sigma\>$ implies $f(H)\vd\<f\tau,f\sigma\>$ for all $f:\F\to\F$. 
	%
 Define  $\vd_{\cL}$ as follows. For all $H\subseteq F\times F$ and $\tau,\sigma\in F$, 
 	\begin{align}\begin{split}	
 H\ \vd_{\cL}\ \<\tau,\sigma\>\ \ \Leftrightarrow\ \ (\forall f&:\F\to\F)(\forall \gM\in M)\\ &\big[f(H)\subseteq \ker(\mng_{\gM})\mbox{ $\Rightarrow$ }\< f\tau,f\sigma\>\in \ker(\mng_{\gM})\big], 
 \end{split}\label{vddef}\end{align}
	where $f(H)=\{\< f(\phi),f(\psi)\> : \<\phi,\psi\>\in H\}$.

The following theorem says that each logic $\cL$ contains a largest substitutional and sound, but not necessarily finitary, (2-dimensional) deductive system in it.	This deductive system is $\vd_{\cL}$ defined above and it corresponds to the infinitary quasi-equational theory of $\Alg_m(\cL)$ such that the atomic formulas are considered as algebraic variables.

\begin{theorem}\label{ded} Let $\cL$ be any logic in the sense of Definition \ref{def:logic}. Then (i)-(ii) below hold.
	\begin{enumerate}[(i)]\itemsep-2pt
		\item $\vd_{\cL}$ is the largest substitutional and sound deductive system in $\cL$.
		\item\label{deq}  $\vd_{\cL}$ corresponds to the infinitary quasi-equational theory of $\Alg_m(\cL)$, i.e., for all $H\subseteq F\times F$ and $\phi,\psi\in F$ we have
		\[ H\vd_{\cL}\<\phi,\psi\>\quad\mbox{ iff }\quad\Alg_m(\cL)\models\wedge H\to [\phi=\psi]. \]
	\end{enumerate}
\end{theorem}

\begin{proof} 
    (i) $\vd_{\cL}$ is sound, if $\ker(\mng_{\gM})$ is closed under $\vd_{\cL}$ for all models $\gM$, that is, for $H\subseteq\ker(\mng_{\gM})$ if $H\vd_{\cL}\<\tau,\sigma\>$, then $\<\tau,\sigma\>\in\ker(\mng_{\gM})$. This follows from \eqref{vddef} with $f$ being the identity mapping.
    The composition $f\circ g$ of two homomorphisms $f,g:\F\to\F$
    is a homomorphism, which yields substitutionality of $\vd_{\cL}$.
    
    By way of contradiction, suppose that $\vd$ is a 
    substitutional and sound deductive system in $\cL$, and 
    $\vd_{\cL}\subsetneq \vd$. Then there are $H\subseteq F\times F$
    and $\tau,\sigma\in F$ such that $H\vd\<\tau,\sigma\>$, but
    $H\nvd_{\cL}\<\tau,\sigma\>$. The latter implies that
    there are $f:\F\to\F$ and $\gM\in M$ so that
    $f(H)\subseteq\ker(\mng_{\gM})$ but $\<f(\tau),f(\sigma)\>\notin\ker(\mng_{\gM})$. Substitutionality of
    $\vd$ ensures $f(H)\vd\<f(\tau),f(\sigma)\>$. By assumption, 
    $f(H)\subseteq\ker(\mng_{\gM})$. Soundness of $\vd$ implies
    $\<f(\tau),f(\sigma)\>\in\ker(\mng_{\gM})$,  yielding a contradiction. Therefore, $\vd_{\cL}$ is the largest among the
    substitutional and sound deductive systems in $\cL$.\\

    \noindent (ii) The proof of Theorem \ref{fmalg}(iii), in particular \eqref{eqeqeq} shows that
     	\begin{align}\begin{split}	
 \Alg_m(\cL)\models \wedge H\to &[\phi=\psi]\ \ \Leftrightarrow\ \ (\forall f:\F\to\F)(\forall \gM\in M)\\ &\big[f(H)\subseteq \ker(\mng_{\gM})\mbox{ $\Rightarrow$ }\< f\tau,f\sigma\>\in \ker(\mng_{\gM})\big].
 \end{split}\label{eqqe}\end{align}
 The right-hand side of \eqref{eqqe} is the definition of
 $H\vd_{\cL}\<\tau,\sigma\>$.
\end{proof}

\bigskip
We finish this subsection with defining the notions of isomorphism between logics and reduct of a logic. In the next two definitions, let $\cL^P=\langle F^P,M^P,\mng^P,\models^P\rangle$ and $\cL^Q=\langle F^Q,M^Q,\mng^Q,\models^Q\rangle$ be logics with atomic formulas $P,Q$ and formula algebras $\F^P, \F^Q$ respectively.. 
\begin{definition}[isomorphism between logics]\label{def:isom}
	An \emph{isomorphism}	between $\cL^P$ and $\cL^Q$ is a pair $f^F, f^M$ of bijections  such that
	\begin{enumerate}[(a)]\itemsep-2pt
		\item $f^F$ is an isomorphism from $\F^P$ onto $\F^Q$, 
		\item $f^M$ is a bijection between $M^P$ and $M^Q$, and
		\item for all $\phi\in F^P$ and $\gM\in M^P$
		\begin{align}
			\mng^P_{\gM}(\phi)&=\mng^Q_{f^M(\gM)}\bigl(f^F(\phi)\bigr),\\
			\gM\models^P\phi\quad &\Longleftrightarrow\quad f^M(\gM)\models^Q
			f^F(\phi).
		\end{align}
	\end{enumerate}
	We say that a bijection $f:P\to Q$ \emph{induces an isomorphism} 
	between $\cL^P$ and $\cL^Q$ if there is an isomorphism $f^F, f^M$ between $\cL^P$ and $\cL^Q$ such that $f^F$ extends $f$.
	We say that $\cL^P$ is an {\it isomorphic copy} of $\cL^Q$ if there
	is an isomorphism between them.	\endef
\end{definition}
Note that isomorphic logics have the same class of meaning algebras, this is immediate from the definition of an isomorphism.

\begin{definition}[reduct and conservative extension of a logic]\label{def:reduct}
	We say that $\cL^P$ is a \emph{reduct} of $\cL^Q$ if there is a function $f^M$  such that
	\begin{enumerate}[(a)]\itemsep-2pt
		\item $P\subseteq Q$ and $\F^P$ is a subalgebra of $\F^Q$, 
		\item $f^M$ maps $M^Q$ onto $M^P$, and
		\item for all $\phi\in F^P$ and $\gM\in M^Q$
		\begin{align}
			\mng^Q_{\gM}(\phi)&=\mng^P_{f^M(\gM)}(\phi),\\
			\gM\models^Q\phi\quad &\Longleftrightarrow\quad f^M(\gM)\models^P(\phi).
		\end{align}
	\end{enumerate}
		We say that $\cL^Q$ is a \emph{conservative extension} of $\cL^P$ when  $\cL^P$ is a reduct of $\cL^Q$.
	We will often omit the adjective ``conservative''. \endef
	\end{definition}

The following proposition says that the tautological congruence of a reduct is the natural restriction of that of the extension, and thus the tautological formula algebra of a reduct is naturally embeddable in the tautological formula algebra of the extension. 

\begin{proposition}\label{prop:reduct} Assume that $\cL^P$ is a reduct of $\cL^Q$. Then (i)-(ii) below hold.
	\begin{enumerate}[(i)]\itemsep-2pt
		\item $\sim^P =\ \sim^Q\cap\ (F^P\times F^P)$
		\item $f:\F^P\slash\!\sim^P\to\F^Q\slash\!\sim^Q$ is an embedding, where $f(\phi\slash\!\sim^P)=\phi\slash\!\sim^Q$ for all $\phi\in F^P$.
	\end{enumerate}	
\end{proposition}

\begin{proof} For proving (i), let $\phi,\psi\in F^P$. Assume that $(\phi,\psi)\notin\ \sim^P$. Then there is $\gM\in M^P$ such that $\mng_\gM(\phi)\ne\mng_\gM(\psi)$. There is $\gN\in M^Q$ such that $\mng_\gM = \mng_\gN\upharpoonright F^P$ since $\cL^P$ is a reduct of $\cL^Q$. Thus $\mng_\gN(\phi)\ne\mng_\gN(\psi)$ and therefore $(\phi,\psi)\notin\ \sim^Q$. The other direction is completely analogous: assume that $(\phi,\psi)\notin\ \sim^Q$.
	Then there is $\gN\in M^Q$ such that $\mng_\gN(\phi)\ne\mng_\gN(\psi)$. Let $\gM\in M^P$ be such that $\mng_\gM = \mng_\gN\upharpoonright F^P$, there is such $\gM$ since $\cL^P$ is a reduct of $\cL^Q$. Thus $\mng_\gM(\phi)\ne\mng_\gM(\psi)$ and therefore $(\phi,\psi)\notin\ \sim^P$. This proves (i). Now, (ii) is a direct consequence of (i).
\end{proof}

\section{Propositional- and predicate-like logics}\label{sec:subs}
In this section, we define two characteristic subclasses of logics. In the first class, atomic formulas can be replaced (substituted) with any other formulas without affecting meaning, in the second class, this can be done only with formulas satisfying some conditions. Propositional logics usually fall into the first class, while classical first-order logic is a typical example of the second class.

The substitution property formalizes a 
``total (or unconditional) freedom'' of the propositional variables and is as
follows.%

\begin{definition}[substitution property]\label{def:substit}
	We say that	$\cL$ has the \emph{substitution property}%
	\footnote{This is called semantical substitution property in \cite[3.3.10]{AGyNS2022}. The syntactic substitution property of \cite[3.3.9]{AGyNS2022} corresponds to the tautological congruence being substitutional, see Proposition \ref{prop:syntsubs}.}
	or that $\cL$ is \emph{substitutional} if
	for any $s:P\to F$ and any model $\gM\in M$ there is a model $\gN\in M$
	such that for all $p\in P$ we have
	\begin{equation}
		\mng_{\gN}(p) = \mng_{\gM}(s(p)).
	\end{equation}
	The model $\gN$ is called the \emph{substituted version} of $\gM$
	along substitution $s$. \endef
\end{definition}

The following is proved in \cite[3.3.11]{AGyNS2022}.

\begin{claim}[equivalent forms of being substitutional]\label{subsequi}
	The following are equivalent.
	\begin{enumerate}[(i)]\itemsep-2pt
		\item $\cL$ is substitutional.
	\item $Mng_{\cL}$ is the class of all homomorphisms from the formula algebra $\F$ to members of $\Alg_m(\cL)$, i.e.,
	\begin{equation}
		Mng_{\cL} = \Hom(\F,\Alg_m(\cL)).
	\end{equation}
		\item There is a class $\K$ of algebras such that
		$\Mng_{\cL}$ is the class of all homomorphisms from the formula algebra to members of $\K$, i.e.,
		\begin{equation}
			Mng_{\cL} = \Hom(\F,\K)\quad\mbox{ for some }\K.
		\end{equation}
	\end{enumerate}
\end{claim}

The substitution property implies that the tautological congruence relation is substitutional, i.e., it is its own substitutional part.

\begin{proposition}[syntactic substitution]\label{prop:syntsubs}
	If $\cL$ has the substitution property, then $\si(\sim_{\cL})=\ \sim_{\cL}$.
\end{proposition}
\begin{proof} Assuming that $\<\phi,\psi\>\in\;\sim_{\cL}$ and $h:\F\to\F$ is a substitution, we have to show that $\< h(\phi),h(\psi)\>\in\;\sim_{\cL}$.	Let $\gM\in M$ be a model, then $h\circ\mng_{\gM}=\mng_{\gN}$ for some model $\gN$, by \ref{subsequi}(ii). Then $\<\phi,\psi\>\in\ker(\mng_{\gN})$ by $\<\phi,\psi\>\in\;\sim_{\cL}$, so  $\< h(\phi),h(\psi)\>\in\ker(\mng_{\gM})$. Since $\gM\in M$ was chosen arbitrarily, this means that $\<\phi,\psi\>\in\;\sim_{\cL}$.
\end{proof}

\begin{example}\label{ex:pcsubst}
	Classical propositional logic 
	$\CPL$ is substitutional (cf. \cite[3.3.13]{AGyNS2022}).
\end{example}
\begin{proof}
	For any set $P$ of propositional variables and 
	any homomorphism $h:\F^P\to\mathbf{2}$, the mapping
	$\gM = h\upharpoonright P$ is a model in $\cL_C^P$. 
	As $P$ generates $\F^P$, the extension $\mng_{\gM}$ of $\gM$
	to $\F^P$ coincides with $h$. Thus, any homomorphism $h:\F^P\to \mathbf{2}$
	is a meaning function of some model. Applying Claim \ref{subsequi}(ii) finishes the proof.	
\end{proof}

\begin{example}\label{claim:FOLnosub}
	First-order logic $\FOL_t(V)$ is not substitutional, whenever $t$ is nonempty and $|V|>1$.
\end{example}
\begin{proof}
	An easy counterexample is the following. Let $r$ be any relation symbol,
	and $x,y\in V$ be distinct variables. In the example below we assume that $r$ is
	 unary, the case $t(r)>1$ is completely analogous. We have
	\begin{eqnarray}
		&\models& r(x)\;\longleftrightarrow\;\exists y(x=y\;\land r(y)) \\
		&\not\models& r(y)\;\longleftrightarrow\;\exists y(x=y\;\land r(y)).
	\end{eqnarray}
	Therefore the substitution that maps $r(x)$ to $r(y)$ and $r(y)$ to $r(y)$
	does not preserve the tautological congruence. Applying Proposition \ref{prop:syntsubs} finishes the proof.
\end{proof}

\begin{remark}\label{remark:FOLuresegyelemu}
	Let us discuss the cases $|V|\leq 1$ concerning \ref{claim:FOLnosub}.
	\begin{itemize}
		\item If $V=\emptyset$, then $P_t(\emptyset)$ is empty for any similarity type $t$. 
	 The set of formulas $F_t(V)$ is generated by $\bot$ and the $\land$, $\lnot$. 
		This is because we did not allow nullary relation symbols.
		The logic $\FOL_t(\emptyset)$ is therefore classical
		propositional logic without propositional letters, and is substitutional by
		\ref{ex:pcsubst}. 
		%
		We note that if nullary relations were allowed, then
		$\FOL_t(\emptyset)$ would be the same as propositional logic $\CPL$.
		\item If $|V|=1$, then $\FOL_t(V)$ is basically the same
		as modal logic $S5$, which is substitutional (see \cite[Example 3.2.5]{AGyNS2022}).
	\end{itemize}
\end{remark}

Though first-order logic is not substitutional, it has an analogous and just as useful substitution property. 
Namely, an atomic formula $r(v_1,\dots,v_n)$ represents any formula with free variables $v_1,\dots,v_n$ in the sense that 
we can substitute $r(v_1,\dots,v_n)$ with any formula with free variables $v_1,\dots,v_n$ in a model, and we get another model. 
In first-order logic as formalized in Example \ref{def:FOL}, 
there is another feature coded in the set of atomic formulas. Though an atomic formula $r(v_1,\dots,v_n)$ 
can take the meaning of that of any formula $\phi(v_1,\dots,v_n)$ with free variables $v_1,\dots,v_n$, there are connections between the meanings of $r(v_1,...,v_n)$ and the same relation symbol followed by a different sequence of variables, say, $r(v_2,v_1,\dots,v_n)$. Namely the meaning of $r(v_2,v_1,...,v_n)$ is determined as soon as we know the meaning of $r(v_ 1,v_2,\dots,v_n)$. 
Luckily, in $\FOL_t(V)$ with $V$ infinite, all these connections can be expressed in the language, and so they are present in the tautological congruence. 
The following weaker substitution property, conditional substitutional property, is intended to capture this feature of a logic.

\begin{definition}[conditional substitution property]\label{def:condsubstit}
	We say that the logic $\cL$ has the 
	\emph{conditional substitution property},
	or that $\cL$ is \emph{conditionally substitutional}, if
	for any homomorphism $s:\F\to\F$ and any model $\gM\in M$,
	\begin{equation} \label{ceq}
		\sim_{\cL}\;\subseteq\;\ker(\mng_{\gM}\circ\;s)\quad\Longrightarrow\quad
		(\exists \gN\in M)(\forall p\in P)\ \mng_{\gN}(p) = \mng_{\gM}(s(p)).
	\end{equation}
	The model $\gN$ is called the \emph{substituted version} of $\gM$
	along the substitution $s$.
\endef
\end{definition}

The following is proved in \cite[3.3.15]{AGyNS2022},

\begin{claim}[equivalent forms of being conditionally substitutional]\label{claim:subschar}
	The following are equivalent.
	\begin{enumerate}[(i)]\itemsep-2pt
		\item $\cL$ is conditionally substitutional.
		\item $\Mng_{\cL} = \Hom(\F,\Alg_m(\cL),\sim_{\cL}).$
		\item $\Mng_{\cL} = \Hom(\F,\K,S)$\quad\mbox{ for some }$\K\mbox{ and }S$.
	\end{enumerate}
\end{claim}

\begin{theorem} \label{thm:FOLconsubs}
	For an infinite set $V$ of variables, first-order logic 
	$\FOL_t(V)$ is conditionally substitutional.
\end{theorem}
\begin{proof} Let us introduce some notation. We fix a one-to-one enumeration $v = \< v_i:\; i\in |V|\>$
of the variables, i.e., $V=\{ v_i : i\in|V|\}$. 
For a relation symbol $r$ of arity $n$ we write
\begin{equation}
	S(r) = \big\{
	\big\< r(v_0,\ldots,v_{n-1}),\ \exists v_j r(v_0,\ldots,v_{n-1})\big\>:\;
j\ge n	\big\}.
\end{equation}
For $r(v_{i_0}, \ldots, v_{i_{n-1}})$ let $r^*(v_{i_0}, \ldots, v_{i_{n-1}})$
denote the formula
\begin{equation}
	\exists v_{j_0}\cdots\exists v_{j_{n-1}}
	\big[ \Land_{k<n}v_{i_k}=v_{j_k}\;\land
	\exists v_0\cdots\exists v_{n-1}( 
	\Land_{k<n}v_{k}=v_{j_k}\;\land r(v_0,\ldots,v_{n-1})) \big] \label{eq:Rstar}
\end{equation}
where the variables $v_{j_0}$, $\ldots$, $v_{j_{n-1}}$
are the first in the enumeration that are not among
$v_0$,$\ldots$, $v_{n-1}$, $v_{i_0}$, $\ldots$, $v_{i_{n-1}}$. Write
\begin{equation}
	E(r) = \big\{
	\big\< r(v_{i_0}, \ldots, v_{i_{n-1}}),\ 
	r^*(v_{i_0}, \ldots, v_{i_{n-1}}) \big\>:\; v_{i_0},\ldots, v_{i_{n-1}}\in V
	\big\}. \label{eq:ER}
\end{equation}
Note that in the definition of $E(r)$ we used that there are infinitely many 
variables available. Finally, for a similarity type $t$ let us write
\begin{equation}
	D(t) = \bigcup\big\{ S(r)\cup E(r):\; r\in\dom(t) \big\}.
\end{equation}
The set $S(r)$ expresses that the relation $r(v_0,\ldots,v_{n-1})$ 
depends only on the variables $v_0$, $\ldots$, $v_{n-1}$,
while $E(r)$ expresses
that the meaning of $r(v_0,\ldots,v_{n-1})$ determines the meaning of
$r(v_{i_0}, \ldots, v_{i_{n-1}})$ for any sequence of variables. 

By rules of first-order logic for any similarity type $t$, the
set $D(t)$ is a subset of the tautological congruence of $\FOL_t(V)$. 

Let $\cL$ denote $\FOL_t(V)$ and $\F$ denote the formula algebra of $\cL$. We are 
going to show that the class of meaning functions of $\cL$ is the class of homomorphisms 
from $\F$ to $\Alg_m(\cL)$ whose kernels contain $D(t)$:
\begin{equation}\label{folmng}\Mng_{\cL}  =  \Hom(\F,\Alg_m(\cL), D(t)).\end{equation}

The inclusion $\subseteq$ in \eqref{folmng} is clear by $D(t)\subseteq\,\sim_{\cL}$. 
To prove the other direction, let $h:\F\to\gA\in\Alg_m(\cL)$ be such that $D(t)\subseteq\ker(h)$, we want to show that $h\in\Mng_\cL$.
Let $\gM\in M_t$ be such that $\gA=\mng_{\gM}(\F)$. 
Let us define the model 
$\gN = \<M, r^{\gN}\>_{r\in\dom(t)}$ as follows: for $r\in\dom(t)$, 
$t(r)=n$ write	
\begin{equation}
	r^{\gN} =\big\{ \<e(v_0),\ldots,e(v_{n-1})\>:\;
	 e\in h(r(v_0,\ldots,v_{n-1}))\big\} . \label{eq:mod1}		
\end{equation}
We claim that $h = \mng_{\gN}$. For this it is enough to check that
\begin{equation}
	\mng_{\gN}(p) = h(p) \label{nn}
\end{equation}
holds for each $p\in P_t(V)$ since $P_t(V)$ generates $\F$. Before proceeding, we introduce a property of $\gA$. For $a\in A$ let
\[\Delta(a) = \{ x\in V : (\exists x)^{\gA}(a)\ne a\} .\]
We say that $a\in A$ is \emph{regular in $\gA$} iff $e\in a$ depends only on its behaviour on $\Delta(a)$, i.e., 
\[ e\in a\mbox{ iff }e'\in a\quad\mbox{ for all $e,e'\in {}^VM$ such that}\ \  e\upharpoonright\Delta(a)=e'\upharpoonright\Delta(a) .\]
Now, it is not difficult to prove that 
\begin{equation}\label{reg} \mbox{all elements of $\gA$ are regular.}\end{equation}
To prove \eqref{nn}, let $r\in\Dom(t)$ and $n=t(r)$. Take first $p=r(v_0,\dots,v_{n-1})$ and let $e:V\to N$. Note that $S(t)\subseteq\ker(h)$ means that
\begin{equation}\label{dp} \Delta(h(p))\subseteq\{ v_0,\dots,v_{n-1}\} .\end{equation}
Now,
\begin{align*}
    e\in\mng_{\gN}(p) &\text{ iff } 
        \<e(v_0),\ldots,e(v_{n-1})\>\in r^{\gN} 
        && (\text{by def. of} \mng_{\gM}) \\
    &\text{ iff } 
        \<e(v_0),\ldots,e(v_{n-1})\>=\<e'(v_0),\ldots,e'(v_{n-1})\> \\
    & \qquad\quad \text{ for some } e'\in h(p)
    && (\text{by def. of } r^{\gN})\\
    &\text{ iff }  e\in h(p)
    && (\text{by} \eqref{reg}, \eqref{dp}).
\end{align*}

\noindent We have shown that $\mng_{\gN}(p)=h(p)$ for $p=r(v_0,\dots,v_{n-1})$. Take now $q=r(v_{i_0},\dots,v_{i_{n-1}})$. Then $q^*$ is a formula depending only on $p$, so $\mng_{\gN}(q^*)=h(q^*)$ by $\mng_{\gN}(p)=h(p)$, since both $\mng_{\gN}$ and $h$ are homomorphisms. We have $\mng_{\gN}(q)=\mng_{\gN}(q^*)$ by $E(t)\subseteq\,\sim_{\cL}$ and also $h(q)=h(q^*)$ by $E(t)\subseteq\ker(h)$. Hence, $\mng_{\gN}(q)=h(q)$. We have show that $\mng_{\gN}$ and $h$ agree on $P_t(V)$, so $\mng_{\gN}=h$, i.e., $h\in\Mng_{\cL}$ and we are done with proving \eqref{folmng}.

Finally, \eqref{folmng} implies that $\FOL_t(V)$ is conditionally substitutional, by \ref{claim:subschar}(iii). 
\end{proof}

\bigskip

\noindent One of the key steps of the proof was to show that if
$V$ is infinite, then every formula of $\FOL_t(V)$ is equivalent to a 
restricted formula (this is expressed
by the sets $E(r)$ in \eqref{eq:ER}). Let us make this idea more precise.
For a similarity type $t$ consider the following set
\begin{equation}
	P_t^r(V) =\big\{  r(v_0,\ldots,v_{n-1}):\; r\in\dom(t),\ t(r)=n  \big\},
\end{equation}
called the set of restricted atomic formulas. The set of \emph{restricted formulas}
$F_t^r(V)$ is generated by $P_t^r(V)$ using the first-order connectives
$\land$, $\lnot$, $\bot$, $\exists v_i$ and $v_i=v_j$ (for $i,j<|V|$). 
Applying an inductive argument using the construction in \eqref{eq:Rstar} we
get that any formula in $F_t(V)$ is equivalent to a restricted formula in $F_t^r(V)$,
provided $V$ is infinite. 

However, if $|V|=\alpha$ is finite and there is a relation 
symbol of arity $\alpha$, then not every formula is equivalent to a restricted one. 
For example, in the case $V=\{v_0, v_1\}$, the non-restricted atomic formula
$r(v_1, v_0)$ is not equivalent to any restricted formula (as permutation of variables
cannot be expressed by quantification and identities without reference to an extra variable), 
see, e.g., the proof of \ref{claim:FOLnnocondsubs} below.

We note that it is not the arity of the relation what matters but whether
the atomic formula uses up all the variables available. As an illustration, below we show
how to express $r(v_1,v_0,v_0)$ with a restricted formula, when $V=\{v_0,v_1,v_2\}$ (thus, we have $3$ variables, 
and $r$ is ternary; but $r$ does not use all the $3$ variables).
\begin{eqnarray*}
	r(v_1, v_0, v_0) &\Longleftrightarrow& \exists v_2( v_2=v_0\land r(v_1,v_2,v_2)) \\
	&\Longleftrightarrow& \exists v_2(v_2=v_0\land \exists v_0(v_0=v_1\land r(v_0,v_2,v_2))) \\
	&\Longleftrightarrow& \exists v_2( v_2=v_0\land \exists v_0( v_0=v_1 \land \exists v_1 (v_1=v_2\land r(v_0,v_1,v_1))  )   ).
\end{eqnarray*}

The assumption that there are infinitely many variables available was crucial in
Thm.\ref{thm:FOLconsubs}. We have seen in \ref{claim:FOLnosub} that $\FOL_t(V)$
is not substitutional (for $|V|>1$). The next claim shows that $\FOL_t(V)$
with a \emph{finite} and at least two-element set $V$ of variables is not even conditionally substitutional. 
We note that \ref{claim:FOLnnocondsubs} below follows, for complex enough $t$, also from Theorems \ref{thm:pwp2}, \ref{thm:pwp1} and \ref{thm:nofamily}.

\begin{example}\label{claim:FOLnnocondsubs}
	If $2\leq |V|<\omega$ and $t$ has at least one $|V|$-ary relation symbol then $\FOL_t(V)$ is not
	conditionally substitutional.
\end{example}
\begin{proof}
	We exhibit a model $\gM$ and a homomorphism $h:\F\to\mng_{\gM}(\F)$ such that the
	kernels of $\mng_{\gM}$ and $h$ coincide, but $h$ is not the meaning function 
	for any model $\gN$. By \ref{claim:subschar}(ii)  this shows that $\FOL_t(V)$ is not conditionally substitutional. 
	We present the example for $|V|=2$ and when $t$ has two binary relation symbols, $R$ and $S$. The model $\gM$ is defined as follows.
	\begin{eqnarray}
		M&=&\{ 0,1,2,3,4\}, \\
		R^{\gM}&=&\{\<u, v\> : v=u+1\mbox{ mod }5\},\\ 
		S^{\gM}&=&\{\< u, v\> : v=u+2\mbox{ mod }5\}.
	\end{eqnarray} 
	Then \begin{equation}
		\mng_{\gM}(R(v_1,v_0))=\{\< u, v\> : v=u+4\mbox{ mod }5\}
	\end{equation} and
	\begin{equation}
		\mng_{\gM}(S(v_1,v_0))=\{\< u, v\> : v=u+3\mbox{ mod }5\}.
	\end{equation}
	Define the homomorphism $h$ so that $h$ agrees 
	with $\mng_{\gM}$ on $R(v_0,v_1)$ and $S(v_0,v_1)$, 
	but $h(R(v_1,v_0))=\mng_{\gM}(S(v_1,v_0))$ and
	$h(S(v_1,v_0))=\mng_{\gM}(R(v_1,v_0))$. Let $\gA$ denote the concept algebra of $\gM$, i.e., $\gA=\mng_{\gM}(\F)$.
	It is not hard to check that there is an automorphism $\alpha$ of $\gA$ which leaves $\mng_{\gM}(R(v_0,v_1))$ and $\mng_{\gM}(S(v_0,v_1))$ fixed and interchanges $\mng_{\gM}(R(v_1,v_0))$ and $\mng_{\gM}(S(v_1,v_0))$. Then $h=\alpha\circ\mng_{\gM}$ since both are homomorphisms and they agree on the atomic formulas. So, $\ker(h)=\ker(\mng_{\gM})$ but $h$ is not the meaning function of any model $\gN$.		
\end{proof}

\noindent Let us now move on to the main topic of the paper.

\section{Logic as a scheme -- logic families}\label{sec:families}
In section \ref{sec:logic}, our example logics were defined in a way that there was a parameter
$P$ which stood for the set of atomic formulas. This parameter was either an arbitrary
set (e.g., in the case of classical propositional logic $\CPL$)
or was given by a similarity type as in the case of first-order logic $\FOL_t(V)$.  
Thus, in first-order logic the parameter $P$ had further structure,
not all sets could be considered to be the set of atomic formulas. 
In the next definition we capture the aspect of a logic where the set of atomic formulas is
a parameter --- we define what could be called a \emph{logic scheme}, reflecting that
the parameter $P$ can vary but the rest of the definition is of the same pattern.
Schemes of logics are called \emph{general logic}s in \cite{AGyNS2022} and  \cite{ANS2001}.
In the present paper we use the terminology logic family.

In subsection \ref{GLdef-subsec}, we define the notions of logic family and pre-family of logics, and we discuss the conditions in the definition.
In subsection \ref{disc4-subsec}, we begin discussing the least natural condition in a logic family, that which concerns union of signatures. 
In subsection \ref{subs:cs4}, we discuss this condition in the special case of conditionally substitutional pre-families and we show that each conditionally substitutional logic is a member of a logic family.

\subsection{Logic families and discussion of logic pre-families}\label{GLdef-subsec}
In what follows we give the definition of a logic family. This definition basically is the same as \cite[3.3.27]{AGyNS2022}, but there are differences which we mention in the discussion that follows. 
Of the conditions given in the definition, conditions (4) and (5) are specific for first-order-like logics, they are not present in the families of logics as specified so far for propositional-like logics. In this subsection we discuss condition (5) and in the next subsections we discuss condition (4).

\begin{definition}[Logic family, logic pre-family]\label{def:generallogic}
	A \emph{logic family} is a family $\LL = \< \cL^P:\; P\in\Sig\>$,
	where $\Sig$ is a class of sets, 
	$\cL^P$
	is a logic in the sense of 
	Definition \ref{def:logic} for each $P\in\Sig$, and conditions (1)--(5) 
	below are satisfied for any sets $P,Q\in\Sig$.
	\begin{enumerate}[(1)]\itemsep-2pt
		\item $P$ is the set of atomic formulas of $\cL^P$.
		\item $\Cn(\cL^P)=\Cn(\cL^Q)$.
		\item If $P\subseteq Q$, then $\cL^P$ is a reduct of $\cL^Q$. 
		\item If $P$ is the disjoint union of $P_i\in\Sig$, $i\in I$,  then
			the tautological formula algebra of $\cL^P$ is the conditionally $\Alg_m(\cL^P)$-free algebra, with conditions
			the union of the tautological congruences of $\cL^{P_i}$, i.e.,
			\begin{equation}
				\F^P\slash\!\sim^P\ =\ \Fr(\Alg_m(\cL^P),\; P,\; \cup_{i\in I}\sim^{P_i}).
							\label{gleq}
				\end{equation}			
				 \item The class of signatures is big enough in the following sense.
			\begin{enumerate}[(a)]\itemsep-2pt
					\item For any $P\in\Sig$ and set $H$, there is a $P'\in\Sig$ such that 
				$P'$ is disjoint from $H$ and $\cL^{P'}$ is an isomorphic copy of $\cL^{P}$.				
				\item For sets $P_i$, $i\in I$ from $\Sig$, where $I$ is nonempty,
				 there exist $P_i'\in \Sig$  such that $P'_i$, $i\in I$ are pairwise disjoint, 
			$\cL^{P_i'}$ is an isomorphic copy of $\cL^{P_i}$ and
			the union $\cup_{i\in I}P_i'$ belongs to $\Sig$.		
		\item $\Sig$ contains at least one nonempty set.
	\end{enumerate}
	\end{enumerate}
	For a class $\LL$ as above, $\sim^P$ denotes the tautological congruence 
	relation of its member logic $\cL^P$.
	We write
	\begin{equation}
		\Alg_m(\LL) = \bigcup_{P\in \Sig}\Alg_m(\cL^P) .
	\end{equation}
	A\emph{ pre-family of logics} is a system of logics which satisfies (1)--(3) and (5), i.e., it may not satisfy (4). A system $\LL$ of logics is called \emph{substitutional} when all its members are substitutional, and likewise $\LL$ is called \emph{conditionally substitutional} when all logics in it are such.
	\endef
\end{definition}

We begin discussing Definition \ref{def:generallogic}. First we clarify that by a family $\LL = \< \cL^P:\; P\in\Sig\>$ we mean a function with domain $\Sig$ such that this function assigns $\cL^P$ to a member $P\in\Sig$.%
\footnote{This function is a class-function, since $\Sig$ is a proper class in logic families, as we will see soon.} Thus, we can write $\LL(P)$ in place of $\cL^P$. 
Condition (1) expresses that the parameter $P\in\Sig$ stands for the set of atomic formulas in the individual logics of the family. Condition (2) is in harmony with the aim that the members of a family of logics differ only in the set of atomic formulas. In this, a family of logics differs from the systems of logics treated in combining logics%
\footnote{For this field see, e.g., Gabbay \cite{Fibring}, Beziau-Coniglio \cite{BC05}.}, where the emphasis is in changing the set of connectives of a logic. 

\paragraph{Discussion of condition (3) concerning reducts.} 
Condition (3) requires that the subset relation in $\Sig$ be reflected in the assigned logics as being a reduct.  This condition is quite extensively used in \cite{AGyNS2022} (e.g., Thms. \cite[4.1.5, 4.2.20, 4.3.6]{AGyNS2022}; cf. Remark \cite[3.3.30]{AGyNS2022}). It expresses that when we add new atomic formulas to a logic in the family, meaning and validity 
of old formulas are not changed: the logic with the extended set of atomic formulas is a conservative extension, see Proposition \ref{prop:reduct}. This expresses a kind of uniformity of logics in the family. 

In \cite[Definition 3.3.27]{AGyNS2022}, a weaker condition is used. Namely, in place of $\cL^P$ being a reduct of $\cL^Q$ when $P\subseteq Q$, only $\cL^P$ being a \emph{sublogic} of $\cL^Q$ is required, which means
	\begin{equation}
	\{\mng^P_{\gM} : \gM\in M^P\} = \{ \mng^Q_{\gM}\upharpoonright F^P:\; \gM\in M^Q\}.
	\label{eq:sublogic}
	\end{equation}
The difference is that in being a sublogic, only the meaning functions are used, validity is not mentioned. Further, in principle, several models may have the same meaning functions, which may prevent the existence of the reduct function $f^M:M^Q\to M^P$ with the required properties. However, the two notions coincide in most cases, because usually the meaning functions determine models and validity in the following sense.

\begin{definition}[referential transparency, meaning-determined]\label{def:reftr} Let $\LL=\langle\cL^P : P\in\Sig\rangle$ be a system of logics. 
	\begin{enumerate}[(i)]\itemsep-2pt
		\item We say that $\LL$ is \emph{referentially transparent} if for all $\gA\in\Alg_m(\LL)$ there is $T\subseteq A$ such that the following hold for all $P\in\Sig$ and $\gM\in M^P$:
		\[\mng_{\gM}(\F^P)\subseteq\gA\ \Rightarrow\ (\forall\phi\in F^P)[\gM\models\phi\Leftrightarrow\mng_{\gM}(\phi)\in T].\] 
		\item We say that a logic $\cL$ is \emph{model-slim}, if $\mng_{\gM}=\mng_{\gN}$ implies that $\gM=\gN$, for all models $\gM,\gN\in M$. We say that $\LL$ is model-slim if $\cL^P$ is model-slim for all $P\in\Sig$.
		\item We say that $\LL$ is \emph{meaning-determined} if it is both referentially transparent and model-slim.	\endef	
\end{enumerate}\end{definition}
Referential transparency of $\LL$ expresses that the meaning functions determine validity of formulas uniformly in all logics of the system.  Most logic families are referentially transparent. In particular, both $\LL_C$ and $\LL_{FOL}$ defined in section \ref{sec:exa} are transparent. More generally, all logic families with the filter property as defined in \cite[3.3.27]{AGyNS2022} are transparent. Similarly, most logics are model-slim, too.  

In \cite[Definition 3.3.27]{AGyNS2022}, we require a condition similar to (3) about isomorphisms:
\begin{enumerate}[(6)]\itemsep-2pt
	\item Any bijection $f:P\to Q$ that extends to a bijection between the 
	tautological formula algebras of $\cL^P$ and $\cL^Q$ induces an 
	isomorphism between $\cL^P$ and $\cL^Q$.
\end{enumerate}
Condition (6) expresses that the tautological formula-algebras of the logics in the system determine their logics, up to isomorphism. 
In a meaning-determined system of logics, isomorphism between member-logics is equivalent to a simpler statement, namely
\begin{equation}\label{simple-iso}
	f^F: \F^P\to\F^Q\mbox{ induces an isomorphism\quad iff \quad} \Mng^P = \{ g\circ f^F : g\in\Mng^Q\}.
\end{equation}
Let us say that $f^F$ induces a meaning-isomorphism between $\cL^P$ and $\cL^Q$ iff $\Mng^P = \{ g\circ f^F : g\in\Mng^Q\}$. Thus, inducing a meaning-isomorphism is weaker than inducing an isomorphism, but they are equivalent in meaning-determined families of logics. Let us formulate the corresponding weaker form of (6) as:
\begin{enumerate}[(6a)]\itemsep-2pt
	\item Any bijection $f:P\to Q$ that extends to a bijection between the 
	tautological formula algebras of $\cL^P$ and $\cL^Q$ induces a 
	meaning-isomorphism between $\cL^P$ and $\cL^Q$.
\end{enumerate}
We omitted condition (6) in this paper, because (6a) follows from conditions (1)-(5) used in this paper, in the case of conditionally substitutional logic families. See Theorem \ref{3-theorem}.  At the same time, in \cite{AGyNS2022}, condition (6) above is used only in case of conditionally substitutional logics, and only in the form of (6a) because in all properties discussed in \cite{AGyNS2022}, of isomorphisms only their meaning-isomorphism parts are used. Therefore, all theorems proved in \cite{AGyNS2022} for logic families are true for the present slightly weaker notion of a logic family, too.

\paragraph{Discussion of condition (5) on signatures.} 
In substitutional logics, the natural thing is that $\Sig$ is the class of all sets not containing proper terms, or all infinite such sets. The role of $\Sig$ is more important in first-order like logics. 

Condition (5)a ensures a flexibility in renaming atomic formulas so that the logic remains unchanged. Condition (5)a rules out, e.g., $\Sig = \{ \{ p_i : i<\alpha\} : \alpha \mbox{ is an ordinal} \}$. In the latter, there are no disjoint sets at all, so neither (5)a nor (5)b is satisfied.

Condition (5)b is the requirement that certain signatures can be merged. In propositional 
logics usually  $P\cup Q\in \Sig$ for $P, Q\in \Sig$ since there the parameter $P$ can be any set. However, not every
logic has this feature: in first-order logic, for example, it might happen that a 
relation symbol appears with different ranks in different similarity 
types. In such case one might not want to have 
the union of the two signatures as a signature. The problem can
be dissolved by taking isomorphic copies, that is, ``renaming'' the atomic formulas.
This is exactly what (5)b expresses: after appropriate ``renamings'' the union of the 
signatures will be a signature.

In formulating (5)b, the condition that $I$ is nonempty cannot be omitted, because without this condition (5)b would imply that the empty set is always in $\Sig$ but we do not want to require this. Indeed, if $I=\emptyset$, then $\bigcup\{ P_i : i\in I\}=\bigcup\emptyset = \emptyset$.

An important feature of $\Sig$ is that it contains arbitrarily large elements. Indeed, let $I$ be a arbitrary non-empty set and let $P\in \Sig$ be nonempty. Such $P$ is provided by (5)c. For $i\in I$ let $P_i=P$. 
By (5)b there are pairwise disjoint $P_i'\in\Sig$ for $i\in I$ such that the union $\cup_{i\in I}P_i'$ belongs to $\Sig$. This union must be as large as $I$.

Condition (5)c is included to ensure arbitrarily large elements of $\Sig$. Without (5)c, we would have two completely different cases: one in which $\Sig\subseteq\{\emptyset\}$ and one in which $\Sig$ has arbitrarily large elements. 
Item (5)c is a natural assumption even though it could be omitted. $\Sig=\{\emptyset\}$ is possible even in the first-order logic case  when the set of variables is empty. (Cf.\ Remark \ref{remark:FOLuresegyelemu}.) The reason we added (5)c to the definition is to avoid adding the line ``assuming $\Sig\neq\{\emptyset\}$'' to numerous statements about logic families. 

We note that in many cases, $\Sig$ is closed under taking subsets. However, this is
not always a feature of $\Sig$. Consider classical first-order logic $\FOL$
(Definition \ref{def:FOL}). 
Given a similarity type $t$, $P_t$ consists of atomic formulas of the form 
$r(v_{1}, \ldots, v_{{t(r)}})$ for a relation symbol $r$.  
Not every subset of $P_t$ is of form $P_{t'}$ for some $t'$ because if $r$ is of rank $n$, then the atomic formulas $r(v_1,\dots,v_n)$ have to be in $P_t$ for all choices $v_1,\dots,v_n\in V$. For example, the singleton $\{ r(v_{1}, \ldots, v_{{t(r)}})\}$ is not of form $P_{t'}$ for any $t$ when $|V|\ge 2$. 
\paragraph{General logics as defined earlier.} We finish this subsection with  making a comparison with the formerly used notion of a logic family that got refined in \cite[3.3.27]{AGyNS2022}. The notion of a logic family has formerly been defined in a somewhat different way in \cite[Def.3.2.16]{AKNS}, \cite[Def.39]{ANS2001}, \cite[Def.22]{BH06}, \cite[p.63]{FJ94c}. All these definitions are basically equivalent with each other, for substitutional logics.
	 Let us recall that definition for comparison with the one presented in this paper. In \cite[Def.3.2.16]{AKNS} a general 
	logic is defined to be an indexed family
	\begin{equation}
		\LL = \< \cL^P:\; P\text{ is a set}\>
	\end{equation}
	where $\cL^P$ is a logic in the sense of Def.\ref{def:logic} and 
	the following further stipulations hold.
	\begin{enumerate}[(i)]\itemsep-2pt
		\item $P$ is the set of atomic formulas of $\cL^P$.
		\item $\Cn(\cL^P)=\Cn(\cL^Q)$.
		\item Any bijection $f:P\to Q$ induces an isomorphism between $\cL^P$ and $\cL^Q$. 
		\item For $P\subseteq Q$, the logic $\cL^P$ is a sublogic of $\cL^Q$.
	\end{enumerate}
	The main striking difference is that $\LL$ is defined on the class all sets rather than on a distinguished class $\Sig$.
	Items (i), (ii) and (iv) are respectively the same as (1), (2) and (3) from Definition \ref{def:generallogic}.%
	\footnote{For simplicity, assume here that all logics discussed are meaning-determined.}
	Condition (5) automatically holds by (iii) and $\Sig$ being all sets. 
	We will see that (4) always holds for substitutional logic families (Corollary \ref{cor:subs}). Therefore, any substitutional general logic in the old sense is a logic family in the new sense, too. 

\subsection{Discussion of condition (4) in the general case} \label{disc4-subsec}
In this subsection, we formulate a stronger version (4b) of (4) and we characterize both by giving several equivalent statements shedding light to different aspects of the intuition behind (4). These equivalent statements prove to be quite useful later, e.g., in showing that (4b) is indeed stronger than (4).

We turn to discussing condition (4) which concerns union of signatures. This is the condition that is most restrictive and needs most explanation. This condition is not needed in all investigations. For example, it is not used in investigations connected to completeness of logics, but it is heavily used in investigations connected to compactness and definability issues. See, e.g., \cite[Sec. 4.3, 4.4]{AGyNS2022} and \cite{GyOz2022,OzPhd}. 
For the rest of this subsection, assume that $P_i,P\in\Sig$ for $i\in I$ and $P=\bigcup\{ P_i : i\in I\}$.

The intuition behind (4) is that the logic 
belonging to the union $P$ does not introduce any new relation between the atomic formulas.  
A bold formulation of this intuition would be 
\begin{equation}
\sim^P \ \ =\ \  \cup_{i\in I} \sim^{P_i}. \tag{4a}
\end{equation} So why is (4) more complicated than this?

The reason is that (4a) holds in very few logics, it does not hold even in classical propositional logic $\CPL$.
Let us take two disjoint sets $Q$ and $R$ from $\Sig$. The 
tautological congruence $\sim^{Q\cup R}$ might not even be generated in $\F^{Q\cup R}$ 
by the union of $\sim^Q$ and $\sim^R$. For instance, in classical propositional logic
 if $Q=\{q\}$ and $R=\{r\}$, then the formula $(q\to r)\lor (r\to q)$
is a tautology, i.e., it is $\sim^{Q\cup R}$ equivalent to $\top$. On the other hand,
it is not generated by tautologies using the propositional letters $q$ or $r$ only.  
Thus, 
\begin{equation}
	\sim^{Q\cup R}\ \ne\  \Cg^{\F^{Q\cup R}}(\sim^{Q}\cup\sim^{R}).
\end{equation}
In the above, and later in the paper, $\Cg^{\gA}(X)$ denotes the \emph{congruence relation of $\gA$ generated by} $X\subseteq A\times A$, i.e.,  $\Cg^{\gA}(X)$ is the least congruence of $\gA$ which contains $X$.

This example explains why our requirement in (4) is not even that the tautological 
congruence of the logic belonging to the union $P=\cup P_i$ is generated as a congruence by 
the union of the tautological congruences of the logics belonging to the $P_i$'s.

In place of (4a), we need to say that $\cL^P$ does not introduce any connection between the atomic formulas of different parts that is not compulsory for all formulas in $\cL^P$. There are several ways of expressing what we mean by compulsory connections. One way is to say that $\<\tau,\sigma\>$ is compulsory when it holds for all formulas in place of the atomic ones. A connection $\<\tau,\sigma\>\in F^P\times F^P$ holds for all formulas in $\cL^P$ when $(h\tau,h\sigma)\in\ \sim^P$ for all substitutions $h:\F^P\to\F^P$. Thus, the set of compulsory connections in this sense is just the substitution-invariant part $\si(\sim^P)$ of $\sim^P$. 
The next condition (4b) expresses that in the logic belonging to the union of the signatures only such compulsory connections hold between atomic formulas coming from different signatures of the union. That is, (4b) says that  $\sim^P$ is generated as a congruence by $\cup_{i\in I}\sim^{P_i}\cup\,\si(\sim^P)$ in $\F^P$,
\begin{equation}
	\sim^P  =  \Cg^{\F^P}(\cup_{i\in I}\sim^{P_i}\cup\,\si(\sim^P)) .\tag{4b}
\end{equation}
Theorem \ref{thm:4b} below lists some statements equivalent to (4b). Among others, it states that (4b) is equivalent to condition (6) in the definition of a general logic in \cite[Definition 3.3.27]{AGyNS2022}. 
 
Condition (4) in the definition of a logic family in the present paper is somewhat weaker than (4b) above, see Theorem \ref{example}. 
Our reason for using in this paper the weaker condition is that it turns out that the weaker (4) suffices in almost all situations, in particular, in the proofs in \cite{AGyNS2022} we always use the weaker (4) in place of the stronger (4b). Further, Theorem \ref{thm:pwp2} may suggest that (4) is more useful than (4b) in logic families.

Before stating Theorem \ref{thm:4b}, we recall condition (6) from \cite[Definition 3.3.27]{AGyNS2022}. The condition says that $\sim^P$ and $\cup_{i\in I}\sim^{P_i}$ generate the same congruences in the $P$-freely generated $\Alg_m(\cL^P)$-free algebra $\Fr(\Alg_m(\cL^P),P)$. We need to clarify this condition, because $\sim^P$ and $\cup_{i\in I}\sim^{P_i}$ are not binary relations on $\Fr(\Alg_m(\cL^P),P)$, they are binary relations on the formula algebra $\F^P$.

The most natural way of transferring a binary relation $R$ on an algebra to its homomorphic image, say by $f$, is to define $f(R)=\{\< f(a),f(b)\> : \< a,b\>\in R\}$.
Using this, the tautological congruences of the formula algebras can be transferred to congruences of the $\Alg_m(\cL)$-free algebra as follows.

For any $Q\in \Sig$ let $\mu^Q:\F^Q\to\Fr(\Alg_m(\cL^Q), Q)$ be the homomorphic extension of the identity mapping
$id:Q\to Q$, i.e., $\mu^Q(\tau) = \tau\slash C$, where $\Fr(\Alg_m(\cL^Q),Q)$ is the factor-algebra of $\F^Q$ by the congruence $C$.  
Define $\backsim^Q$ as the $\mu^Q$-image of $\sim^Q$, that is,
\begin{equation}
	\backsim^Q = \{ \<\mu^Q(\phi), \mu^Q(\psi)\>:\; \phi\sim^Q\psi\}.\label{eq:kuu}
\end{equation}
 Observe the difference between the symbols $\sim$ and $\backsim$. 
The surjective homomorphic image of a congruence is not necessarily a congruence, as transitivity can be violated.%
\footnote{In general, surjective 
homomorphic images of congruences are only tolerance relations: reflexive, symmetric and compatible.}
 However, in the present case $\backsim^Q$ is a congruence, because the kernel of $\mu^Q$ is contained in $\sim^Q$ as $\sim^Q$ is defined by the intersection of kernels of meaning homomorphisms, which are homomorphisms from the free algebra to members of $\Alg_m(\cL^Q)$. Indeed, an alternative way would be to define $\backsim^Q$ as follows.
As $\Fr(\Alg_m(\cL^Q), Q)$ has the universal mapping property with respect to the class 
$\Alg_m(\cL^Q)$, for any model $\gM\in M^Q$ there is a homomorphism $m_{\gM}^Q:\Fr(\Alg_m(\cL^Q), Q)\to \mng_{\gM}^Q(\F^Q)$
such that $m_{\gM}^Q\circ \mu^Q = \mng_{\gM}^Q$, i.e., the diagram below commutes.

\begin{center}
	\begin{tikzcd}[column sep=small]
		\F^Q \arrow{rr}{\mu^Q} \arrow[swap]{dr}{\mng_{\gM}^Q}& &\Fr(\Alg_m(\cL^Q), Q) \arrow{dl}{m_{\gM}^Q}\\
		& \mng_{\gM}^Q(\F^Q) & 
	\end{tikzcd} 
\end{center}  

\noindent It is not hard to check that 
\begin{equation}
	\mu^Q(\phi)\backsim^Q\mu^Q(\psi)\;\text{ if and only if }\; (\forall \gM\in M^Q)\;
	m^Q_{\gM}(\mu^Q(\phi)) = m^Q_{\gM}(\mu^Q(\psi)),
\end{equation} 
that is, $\backsim^Q$ is $\bigcap_{\gM\in M^Q}\ker(m_{\gM}^Q)$. Notice
that $\phi\sim^R\psi$ if and only if $\mu^R(\phi)\backsim^R\mu^R(\psi)$.

\bigskip

We turn to characterizing conditions (4b) and (4). By characterizing we mean listing equivalent conditions that highlight different aspects. The notion of a $\K$-free product of algebras will be used in the characterizations. A free product is a special conditionally free algebra, see \cite[end of 0.4.65]{HMT}, it is introduced in \cite{Los}. Let $\K$ be a class of similar algebras and let $\gA_i$ for $i\in I$ be algebras. An algebra $\gA$ is the \emph{$\K$-free product of $\langle\gA_i : i\in I\rangle$} if the following three conditions hold: (a) $\gA_i$ is a subalgebra of $\gA$, for all $i\in I$, (b) $\gA$ is generated by the union of the universes of $\gA_i, i\in I$ and (c) for any $\gB\in\K$ any family of homomorphisms $f_i:\gA_i\to\gB$ can be extended to a homomorphism $f:\gA\to\gB$ such that $f\supseteq\cup\{f_i : i\in I\}$. It can be checked that $\gA$ is unique up to isomorphism when $\gA\in\K$ is required.

We begin with characterizing the stronger condition. The following theorem lists some conditions equivalent to (4b). In it, \eqref{4bsub} is the same as (4b) and \eqref{4balg} is the same as (6) in \cite[3.3.27]{AGyNS2022}. Note that \eqref{4bsub} is a condition that is formulated in terms of syntactic-oriented notions, without the use of $\Alg_m(\cL)$. 

\begin{theorem}[characterization of (4b)]\label{thm:4b} 	Let $\LL=\<\cL^P:\;P\in \Sig\>$ be a pre-family of logics. Let $P_i\in \Sig$ be sets 
	for $i\in I$ such that the union $P=\bigcup_{i\in I}P_i$ belongs to $\Sig$.
Conditions (i)-(iv) below are equivalent. Further, (i) holding for all $P_i\in\Sig$ such that $\cup_{i\in I}P_i\in\Sig$ is equivalent to \eqref{4bsepb} holding for all such $P_i, i\in I$, and to \eqref{4bprodb} holding for all such $P_i, i\in I$.
\begin{enumerate}[(i)]\itemsep-2pt
	\item \label{4bsub}
	$\sim^P = \Cg^{\F^P}(\cup_{i\in I}\sim^{P_i}\cup\,\si(\sim^P))$.
	\item \label{4balg}
	$\Cg^{\Fr(\Alg_m(\cL^P),P)}(\backsim^{P})\ =\ 
	\Cg^{\Fr(\Alg_m(\cL^P),P)}(\cup_{i\in I}\backsim^{P_i})$.
		\item \label{4bsepa}
	$\big(\forall h\in \Hom(\F^P,\HSP\Alg_m(\cL^P)) \big)\big( 
	\;\cup\!\sim^{P_i}\subseteq \ker(h)\ \; \Longrightarrow\;\  
	\sim^P\subseteq\ker(h)\; \big)$.
		\item  \label{4bproda}
	$\F^P\slash\!\sim^P$ is an $\HSP\Alg_m(\cL^P)$-free product of the $\F^{P_i}\slash\!\sim^{P_i}$ for $i\in I$.
	\item \label{4bsepb}
	$\big(\forall h\in \Hom(\F^P,\HSP\Alg_m(\LL)) \big)\big( 
	\;\cup\!\sim^{P_i}\subseteq \ker(h)\ \; \Longrightarrow\;\  
	\sim^P\subseteq\ker(h)\; \big)$.
	\item  \label{4bprodb}
	$\F^P\slash\!\sim^P$ is an $\HSP\Alg_m(\LL)$-free product of the $\F^{P_i}\slash\!\sim^{P_i}$ for $i\in I$.
	\end{enumerate}
\end{theorem}

\begin{proof} Note first that $\cup_{i\in I}\sim^{P_i}\subseteq\ \sim^P$ by Proposition \ref{prop:reduct}(i) and Definition \ref{def:generallogic}(3), and thus $\cup_{i\in I}\backsim^{P_i}\subseteq\ \backsim^P$.  
We start with a general lemma that we will make use of in this proof and
also in the proof of Theorem \ref{gen4} below.

    \begin{lemma}\label{lemma:aux1}
        For any class $\K$ of similar algebras, (a) and (b)
        below are equivalent.
        \begin{enumerate}[(a)]\itemsep-2pt
            \item $\big(\forall h\in \Hom(\F^P,\K) \big)\big( 
		      \;\cup\!\sim^{P_i}\subseteq \ker(h)\ \; \Longrightarrow\;\  
		      \sim^P\subseteq\ker(h)\; \big)$.
            \item $\F^P\slash\!\sim^P$ is a $\K$-free product of the 
                $\F^{P_i}\slash\sim^{P_i}$ for $i\in I$.
        \end{enumerate}
    \end{lemma}
    \begin{proof}[of Lemma \ref{lemma:aux1}]
        Proof of (a)$\Rightarrow$(b): By Proposition \ref{prop:reduct}(i), $\F^{P_i}\slash\!\sim^{P_i}$ for $i\in I$ can be considered as subalgebras of $\F^{P}\slash\!\sim^{P}$ that generate $\F^{P}\slash\!\sim^{P}$. 
        Take any $\gB\in \K$ and homomorphisms 
        $f_i:\F^{P_i}\slash\!\sim^{P_i}\;\to\gB$. 
        Let $q_i:\F^{P_i}\to \F^{P_i}\slash\!\sim^{P_i}$ be the quotient mappings. 
        By the universal mapping property of $\F^{P_i}$ there are $h_i:\F^{P_i}\to\gB$
        such that $h_i=f_i\circ q_i$ for $i\in I$.
        \begin{center}
	    \begin{tikzcd}[column sep=large]
		\F^{P_i} \arrow[dashed]{dr}{h_i} \arrow[swap]{d}{q_i} &  \\
		 \F^{P_i}\slash\!\sim^{P_i}\arrow[swap]{r}{f_i} & \gB
	    \end{tikzcd} \qquad\quad
	    \begin{tikzcd}[column sep=large]
		\F^{P} \arrow{dr}{h} \arrow[swap]{d}{q} &  \\
		 \F^{P}\slash\!\sim^{P}\arrow[swap,dashed]{r}{f} & \gB
	    \end{tikzcd} 
        \end{center}  
        Let $h:\F^P\to\gB$ be the homomorphic extension of $\cup_{i\in I}h_i$. As
        $\sim^{P_i}\subseteq\ker(h_i)$ for each $i\in I$, it follows that
        $\cup_{i\in I}\sim^{P_i}\subseteq\ker(h)$. By (a), 
        $\ker(q)=\;\sim^P\subseteq\ker(h)$. Thus, there is $f:\F^P\slash\!\sim^P\to\gB$
        such that $f\circ q = h$. Clearly $\cup\{f_i: i\in I\}\subseteq f$.\\

        Proof of (b)$\Rightarrow$(a): Take any $h:\F^P\to\gA\in\K$ 
        and suppose $\cup_{i\in I}\sim^{P_i}\subseteq\ker(h)$. Write 
        $h_i=h\upharpoonright F^P_i$, \ $h_i:\F^{P_i}\to\gA$. Clearly $\sim^{P_i}\subseteq\ker(h_i)$, and thus (by the second isomorphism 
        theorem, see e.g. \cite[Thm. 3.5]{Bergman}) there are surjective homomorphisms $f_i:\F^{P_i}\slash\!\sim^{P_i}\to\gA$. Using (b) and properties of the
        free product, there is a surjective $f:\F^{P}\slash\!\sim^{P}\to\gA$ such 
        that $f\supseteq\cup\{f_i:i\in I\}$. If $q:\F^P\to\F^P\slash\!\sim^P$ is the quotient mapping, then $f\circ q = h$, showing that $\sim^P\subseteq\ker(h)$. 
    \end{proof}

    \noindent Let us continue the proof of Theorem \ref{thm:4b}.
    The equivalences (iii)$\Leftrightarrow$(iv) and 
    (v)$\Leftrightarrow$(vi) follow from Lemma
    \ref{lemma:aux1} with $\K$ being $\HSP\Alg_m(\cL^P)$ in the first 
    equivalence, and $\HSP\Alg_m(\LL)$ in the second.
    In what follows, we write $\K=\Alg_m(\cL^P)$ and
    $\Fr^P = \Fr(\Alg_m(\cL^P), P)$. \\

    Proof of (ii)$\Rightarrow$(iii): Let $h:\F^P\to\gA\in\HSP\K$ be
    a homomorphism and write $\mu:\F^P\to\Fr^P$ for the quotient mapping 
    (cf. the lines above \eqref{eq:kuu}). $\Fr^P$ has the universal
    mapping property with respect to the class $\HSP\K$, thus there
    is $f:\Fr^P\to\gA$ such that $h = f\circ \mu$. Now, if
    $\cup_{i\in I}\sim^{P_i}\subseteq\ker(h)$, then 
    $\Cg^{\Fr^{P}}(\cup_{i\in I}\backsim^{P_i})\subseteq \ker(f)$. By (ii)
    then $\backsim^{P}\subseteq\ker(f)$ and thus $\sim^P\subseteq \ker(h)$.\\

Proof of (i)$\Rightarrow$(ii): We prove the following more general statement.
Suppose $\Psi\subseteq \Pi$ are binary relations over an algebra $\gA$ and 
$\Sigma$ is a congruence of $\gA$. Let $q:\gA\to\gA/\Sigma$ be the quotient mapping. Then
\begin{align}
    \Cg^{\gA}(\Pi)=\Cg^{\gA}(\Psi\cup \Sigma) \qquad\Rightarrow\qquad
    \Cg^{\gA/\Sigma}(q(\Pi)) = \Cg^{\gA/\Sigma}(q(\Psi)). \label{eq:4.3i.1}
\end{align}
From Theorem \ref{fmalg}(ii) we know that $\Fr^P$ equals
$\F^P\slash\si(\sim^P)$, therefore applying \eqref{eq:4.3i.1} with $\gA=\F^P$,
$\Pi=\;\sim^{P}$, $\Sigma=\si(\sim^P)$, and $\Psi = \cup_{i\in I}\sim^{P_i}$ yields
the proof of (i)$\Rightarrow$(ii). Let us prove \eqref{eq:4.3i.1}. Assume that
$\Cg^{\gA}(\Pi)=\Cg^{\gA}(\Psi\cup \Sigma)$ holds. The inclusion 
$\Cg^{\gA/\Sigma}(q(\Pi)) \supseteq \Cg^{\gA/\Sigma}(q(\Psi))$ is straightforward
from the assumption $\Psi\subseteq\Pi$. For the converse inclusion pick 
$x,y\in A$ and assume
$\<x/\Sigma, y/\Sigma\>\in \Cg^{\gA/\Sigma}(q(\Pi))$. Then $\<x,y\>\in\Cg^{\gA}(\Pi)$ and thus $\<x,y\>\in \Cg^{\gA}(\Psi\cup \Sigma)$. Universal algebra
shows (see e.g. \cite[Exercise 3.9/2]{Bergman})
\[  q\big(\Cg^{\gA}(\Psi\cup\Sigma)\big)\ =\ \Cg^{\gA/\Sigma}(q(\Psi)).\]
Therefore $\<x/\Sigma, y/\Sigma\>\in \Cg^{\gA/\Sigma}(q(\Psi))$ as desired.\\

Proof of (iv)$\Rightarrow$(i): As we noted at the very beginning of this proof, 
$\cup_{i\in I}\sim^{P_i}\subseteq\;\sim^{P}$, and by definition 
$\si(\sim^{P})\subseteq\;\sim^{P}$. Hence, the inclusion $\supseteq$ in (i) 
is straightforward. To prove the reverse inclusion, take an arbitrary congruence
$\theta \supseteq \cup_{i\in I}\sim^{P_i}\cup\,\si(\sim^P)$. We need to show that
$\sim^P\subseteq\theta$. 
Using that $\si(\sim^{P_i})\subseteq\theta$, by the second isomorphism theorem, $\F^P\slash\theta$ is a homomorphic image of $\F^P\slash\si(\sim^P)$, which latter algebra is the free algebra $\Fr^P$ (by Theorem \ref{fmalg}(ii)). In particular, $\F^P\slash\theta$ belongs to $\HSP\K$. 
Similarly, using that $\sim^{P_i}\subseteq\theta$ for $i\in I$, there are surjective
homomorphisms $f_i:\F^{P_i}\slash\!\sim^{P_i}\to\F^P\slash\theta$. Let 
$f:\F^P\slash\!\sim^{P}\to\F^P\slash\theta$ 
be the homomorphism from the free product property satisfying $f\supseteq\cup\{f_i:i\in I\}$. As the $\F^{P_i}\slash\!\sim^{P_i}$, taken as subalgebras, generate $\F^P\slash\!\sim^{P}$ (by the free product property), $f$ must be surjective. Therefore, $\sim^P\subseteq\theta$, as desired. 
\begin{center}
	\begin{tikzcd}[column sep=large]
		\F^{P} \arrow{r}{q} \arrow[swap]{d}{\mu} & \F^P\slash\theta \\
		 \F^{P}\slash\si(\sim^{P})\arrow[swap,dashed]{ur}{ } & 
	\end{tikzcd} \qquad\quad
	\begin{tikzcd}[column sep=large]
		\F^{P_i} \arrow{r}{q} \arrow[swap]{d}{\mu_i} & \F^P\slash\theta \\
		 \F^{P_i}\slash\!\sim^{P_i}\arrow[swap,dashed]{ur}{f_i} & 
	\end{tikzcd} 
\end{center}  
\bigskip

Finally, that (iii) holding for all $P_i\in\Sig$ such that $P=\cup P_i\in\Sig$
is equivalent to (v) holding for all such $P_i$, $i\in I$ follows from the lemma below with the choice $\TT=\HSP$. 

\begin{lemma}\label{lemma:aux3}
    Let $\TT\in\{\SP,\HSP\}$ be a class operator. The following are
    equivalent.
    \begin{enumerate}[(a)]\itemsep-2pt
        \item For every $P_i\in\Sig$ for $i\in I$ such that $P=\cup_{i\in I}P_i\in\Sig$, 
        \[
            \big(\forall h\in \Hom(\F^P,\TT\Alg_m(\cL^P)) \big)\big( 
	\;\cup\!\sim^{P_i}\subseteq \ker(h)\ \; \Longrightarrow\;\  
	\sim^P\subseteq\ker(h)\; \big)
        \]
        \item For every $P_i\in\Sig$ for $i\in I$ such that $P=\cup_{i\in I}P_i\in\Sig$, 
        \[
            \big(\forall h\in \Hom(\F^P,\TT\Alg_m(\LL)) \big)\big( 
	\;\cup\!\sim^{P_i}\subseteq \ker(h)\ \; \Longrightarrow\;\  
	\sim^P\subseteq\ker(h)\; \big)
        \]
    \end{enumerate}
\end{lemma}
\begin{proof}[of Lemma \ref{lemma:aux3}]
    The direction (b)$\Rightarrow$(a) is immediate from $\Alg_m(\cL^P)\subseteq\Alg_m(\LL)$. Conversely, for the (a)$\Rightarrow$(b) direction, let $h:\F^P\to\gA\in\TT\Alg_m(\LL)$ be such that $\cup_{i\in I}\sim^{P_i}\subseteq \ker(h)$. There are $Q_j''\in\Sig$ 
    and $\gA_j\in\Alg_m(\cL^{Q_j''})$ for $j\in J$ such that 
    $\gA\in\TT\{\gA_j:j\in J\}$.
    Without loss of generality we can assume that $I=J$.
    By Def. \ref{def:generallogic}(5), for each $i\in I$ there are 
    $P_i', Q_i'\in\Sig$ such that $P_i'\cup Q_i'\in \Sig$ and 
    $P_i', Q_i'$ are isomorphic copies respectively
    of $P_i$ and $Q_i''$. Apply \ref{def:generallogic}(5) again to 
    $P_i'\cup Q_i'$ to get $Q_i\in\Sig$ such that 
    $\cup_{i\in I}Q_i\in\Sig$ and $Q_i$ contains an
    isomorphic copy of $P_i$ and $Q_i''$. To simplify notation we do 
    not denote the isomorphisms between these respective logics, 
    but instead we can assume that for $i\in I$,
    \[  Q_i\in \Sig, \quad 
    P_i\subseteq Q_i, \quad
    Q=\cup_{i\in I}Q_i\in\Sig,\quad \gA\in\TT\Alg_m(\cL^Q) .\]
    Let $h_i=h\upharpoonright F^{P_i}:\F^{P_i}\to\gA$. By assumption, 
    $\sim^{P_i}\subseteq \ker(h_i)$. As $P_i\subseteq Q_i$ and $\F^{Q_i}$ is an
    absolutely free algebra, there are $k_i:\F^{Q_i}\to\gA$ such that
    $\sim^{Q_i}\subseteq\ker(k_i)$ and $k_i\upharpoonright F^{P_i} = h_i$ for $i\in I$.
    Let $k:\F^{Q}\to\gA$ be such that $k\supseteq\cup\{k_i:i\in I\}$. Then 
    \[k: \F^Q\to\TT\Alg_m(\cL^Q), \text{ and }  \cup_{i\in I}\sim^{Q_i}\subseteq\ker(k),\]
    and we assumed that (a) holds whenever $Q_i\in\Sig$ are such that 
    $\cup Q_i\in\Sig$. Therefore, $\sim^{Q}\subseteq\ker(k)$. Notice that
    $k\upharpoonright F^P = h$, and by Proposition \ref{prop:reduct}, 
    $\sim^{P}= \sim^{Q}\upharpoonright F^P$. Hence, $\sim^P\subseteq\ker(h)$.
\end{proof}
The proof of Theorem \ref{thm:4b} is completed.
\end{proof}

\noindent Theorem \ref{thm:4b}\eqref{4bsub} yields that condition (4) is not needed in the substitutional case.

\begin{corollary}\label{cor:subs}
	Substitutional pre-families of logics are always families of logics with the stronger condition (4b).
\end{corollary}

\begin{proof} 
That $\sim^P=\si(\sim^P)$ when $\cL^P$ is substitutional, by Proposition \ref{prop:syntsubs}, together with that $\cup_{i\in I}\sim^{P_i}\subseteq\ \sim^P$, by Proposition \ref{prop:reduct}(i) and Definition \ref{def:generallogic}(3), immediately imply (4b).
\end{proof}

We turn to formulating some conditions equivalent to the weaker condition (4). Recall the notion of the largest sound and substitutional deductive system $\vd_{\cL}$ belonging to a logic $\cL$, from above Theorem \ref{ded}. Note that \eqref{4subd} is analogous to Theorem \ref{thm:4b}\eqref{4bsub} and \eqref{4} is equivalent to (4) by Theorem \ref{fmalg}(i).

\begin{theorem}[characterization of (4)]\label{gen4}	Let $\LL=\<\cL^P:\;P\in \Sig\>$ be a pre-family of logics. Let $P_i\in \Sig$ be sets 
	for $i\in I$ such that the union $P=\bigcup_{i\in I}P_i$ belongs to $\Sig$.
	Conditions \eqref{4subd}-\eqref{4prodaa} are equivalent. Further, (i) holding for all $P_i\in\Sig$ such that $\cup_{i\in I}P_i\in\Sig$ is equivalent to \eqref{4homb} holding for all such $P_i, i\in I$, and to \eqref{4prodb} holding for all such $P_i, i\in I$.
	\begin{enumerate}[(i)]\itemsep-2pt
	\item \label{4subd}
	$\sim^P$ is the $\vd_{\cL^P}$-closure of $\cup_{i\in I} \sim^{P_i}$.
	\item \label{4}
		$\Fr(\Alg_m(\cL^P),\; P,\; \cup\!\sim^{P_i})\ =\ 
		\Fr(\Alg_m(\cL^P),\; P,\; \sim^{\cup P_i})	$.
		\item \label{4homa}
		$\big(\forall h\in \Hom(\F^P,\Alg_m(\cL^P)) \big)\big( 
		\;\cup\!\sim^{P_i}\subseteq \ker(h)\ \; \Longrightarrow\;\  
		\sim^P\subseteq\ker(h)\; \big)$.
		\item \label{4homaa}
		$\big(\forall h\in \Hom(\F^P,\SP\Alg_m(\cL^P)) \big)\big( 
		\;\cup\!\sim^{P_i}\subseteq \ker(h)\ \; \Longrightarrow\;\  
		\sim^P\subseteq\ker(h)\; \big)$.
		\item \label{4proda}
		$\F^P\slash\!\sim^P$ is an $\Alg_m(\cL^P)$-free product of the $\F^{P_i}\slash\!\sim^{P_i}$ for $i\in I$.
			\item \label{4prodaa}
		$\F^P\slash\!\sim^P$ is an $\SP\Alg_m(\cL^P)$-free product of the $\F^{P_i}\slash\!\sim^{P_i}$ for $i\in I$.
		\item \label{4homb}
		$\big(\forall h\in \Hom(\F^P,\SP\Alg_m(\LL)) \big)\big( 
		\;\cup\!\sim^{P_i}\subseteq \ker(h)\ \; \Longrightarrow\;\  
		\sim^P\subseteq\ker(h)\; \big)$.
		\item \label{4prodb}
		$\F^P\slash\!\sim^P$ is an $\SP\Alg_m(\LL)$-free product of the $\F^{P_i}\slash\!\sim^{P_i}$ for $i\in I$.
	\end{enumerate}
\end{theorem}

\begin{proof} 
    Note first that $\cup_{i\in I}\sim^{P_i}\subseteq\ \sim^P$ by Proposition \ref{prop:reduct}(i) and Definition \ref{def:generallogic}(3).\\
    
    Proof of (i)$\Leftrightarrow$(iii): Let $h:\F^P\to\gA\in\Alg_m(\cL^P)$
    be a homomorphism. We claim first that $\ker(h)$ is $\vd_{\cL^P}$-closed.
    \begin{claim}
        $\ker(h)$ is $\vd_{\cL^P}$-closed for every 
        $h:\F^P\to\gA\in\Alg_m(\cL^P)$.
        \label{claim:aux1}
    \end{claim}
    \begin{proof}[of Claim \ref{claim:aux1}]
    Recall that $\vd_{\cL^P}$-closedness of $\ker(h)$ means that
    if $H\subseteq\ker(h)$ and $H\vd_{\cL^P}\<\tau,\sigma\>$, i.e. 
   	\begin{align}
         (\forall f:\F^P\to\F^P)(\forall \gN\in M^P)\big[f(H)\subseteq \ker(\mng_{\gN}) 
         \Rightarrow\< f\tau,f\sigma\>\in \ker(\mng_{\gN})\big], 
    \label{vddeff}\end{align}
    then $\<\tau,\sigma\>\in\ker(h)$.
    Let $\gM\in M^P$ be such that $\gA=\mng_{\gM}(\F^P)$. There is 
    a homomorphism $s:\F\to\F$ such that 
	$h = \mng_{\gM}\circ\; s$. To obtain such $s$, note that 
	for each $p\in P$ there is $\psi_p\in F^P$ such that 
	$h(p)=\mng_{\gM}(\psi_p)$. Take then
	$s$ to be the homomorphic extension of the mapping $p\mapsto\psi_p$. 
	\begin{center}
		\begin{tikzcd}[column sep=large]
			\F^P \arrow{r}{h} \arrow[swap,dashed]{d}{s} & \mng_{\gM}(\F^P) \\
			\F^P \arrow[swap]{ur}{\mng_{\gM}} & 
		\end{tikzcd} 
	\end{center}  
    Let now $H\subseteq\ker(h)$ and assume $H\vd_{\cL^P}\<\tau,\sigma\>$.
    Then \eqref{vddeff} with $f=s$ and $\gN=\gM$ gives
   	\begin{align*}
         s(H) \subseteq \ker(\mng_{\gM}) 
         \quad\Rightarrow\quad\< s\tau,s\sigma\>\in \ker(\mng_{\gM}), 
    \end{align*}
    that is, 
   	\begin{align*}
         H \subseteq \ker(\mng_{\gM}\circ s) 
         \Rightarrow\< \tau, \sigma\>\in \ker(\mng_{\gM}\circ s).
    \end{align*}
    As $\mng_{\gM}\circ s$ equals $h$, and $H\subseteq\ker(h)$ by assumption,
    we get $\<\tau,\sigma\>\in\ker(h)$.
    \end{proof}

    \noindent Let us continue proving the equivalence (i)$\Leftrightarrow$(iii).
    Assume (i) and let $h:\F^P\to\Alg_m(\cL^P)$ be such that
    $\cup\sim^{P_i}\subseteq\ker(h)$. As $\ker(h)$ is $\vd_{\cL^P}$-closed, 
    by Claim \ref{claim:aux1}, and $\sim^P$ is the smallest $\vd_{\cL^P}$-closed
    set containing $\cup\sim^{P_i}$, we get $\sim^P\subseteq\ker(h)$. For the converse direction, assume (iii) holds. It is enough to show that
    \begin{align}
    \cup\sim^{P_i}\ \vd_{\cL^P}\;\<\tau,\sigma\> \quad\text{ for each }
    \<\tau,\sigma\>\in\;\sim^P.\label{lev:1}
    \end{align}
    Indeed, if $\cup\!\sim^{P_i}\;\subseteq\Sigma$ is closed, then \eqref{lev:1}
    ensures that $\sim^P\subseteq\Sigma$. To prove \eqref{lev:1}, 
    let $\<\tau,\sigma\>\in\sim^P$ and take
    arbitrary $f:\F^P\to\F^P$ and $\gM\in M^P$. We need to check that
    \begin{align}
        \text{if }\ \cup\sim^{P_i}\subseteq\ker(\mng_{\gM}\circ f), 
        \text{ then } \ 
        \<\tau,\sigma\>\in\ker(\mng_{\gM}\circ f).\label{lev:2}
    \end{align}
    But $\mng_{\gM}\circ f$ is a homomorphism from $\F^P$ into $\Alg_m(\cL^P)$, 
    and thus \eqref{lev:2} is ensured by (iii).\\

    To prove the rest of the equivalences we make use of the following
    lemma.
    \begin{lemma}\label{lemma:aux2}
        For any class $\K$ of similar algebras, (a) and (b)
        below are equivalent.
        \begin{enumerate}[(a)]\itemsep-2pt
            \item $\Fr(\K,\; P,\; \cup\!\sim^{P_i})\ =\ 
                \Fr(\K,\; P,\; \sim^{\cup P_i})$.
            \item $\big(\forall h\in \Hom(\F^P,\K) \big)\big( 
		      \;\cup\!\sim^{P_i}\subseteq \ker(h)\ \; \Longleftrightarrow\;\  
		      \sim^P\subseteq\ker(h)\; \big)$.
        \end{enumerate}
    \end{lemma}
    \begin{proof}[of Lemma \ref{lemma:aux2}]
        (b) is equivalent to $\Hom(\K,P,\cup\sim^{P_i}) = \Hom(\K,P,\sim^{P})$, while (a) is equivalent to
        $\Cr(\K,P,\cup\sim^{P_i}) = \Cr(\K,P,\sim^{P})$.
        The statement follows from that $\Fr(\K,P,S) = \F^P\slash\Cr(\K,P,S)$ and
        $\Cr(\K,P,S) = \cap\{\ker(h):\; h\in\Hom(\K,P,S)\}$.
    \end{proof}

    \noindent With the choice $\K=\Alg_m(\cL^P)$, Lemma \ref{lemma:aux2}
    gives (ii)$\Leftrightarrow$(iii) and Lemma \ref{lemma:aux1} gives (iii)$\Leftrightarrow$(v).
    Recall that
    $\Fr(\K,X,S)=\Fr(\SP\K,X,S)$ (see e.g. \cite[2.5]{AGyNS2022} or \cite[0.4.65]{HMT}).
    Hence, by letting $\K=\SP\Alg_m(\cL^P)$, Lemma \ref{lemma:aux2} yields
    (ii)$\Leftrightarrow$(iv) while Lemma \ref{lemma:aux1} gives (iv)$\Leftrightarrow$(vi). The equivalence (vii)$\Leftrightarrow$(viii) follows from Lemma \ref{lemma:aux1} with $\K=\SP\Alg_m(\LL)$.
    Lemma \ref{lemma:aux3} with $\TT=\SP$ infers the equivalence (vi)$\Leftrightarrow$(vii).
\end{proof}
\bigskip

\noindent Comparing Theorems \ref{thm:4b}\eqref{4bsepa} and \ref{gen4}\eqref{4homa} immediately yields that (4b) implies (4). The following theorem states that (4) is indeed weaker than (4b). We will see later that the natural pre-family formed from $\FOL_t(V)$ when $2\le |V|<\omega$ does not satisfy even condition (4), see Theorem \ref{thm:nofamily}. These show that the conditions (4a), (4b), (4) are strictly weaker in this order, and (4) is still not vacuous.
Thus we have
\[ \qquad (4a) \qquad\underset{\not\!\Leftarrow}{\Rightarrow}\qquad (4b) \qquad\underset{\not\!\Leftarrow}{\Rightarrow}\qquad (4)\qquad.   \]

\begin{theorem}[(4) is weaker than (4b)]\label{example} 
	There is a conditionally substitutional pre-family of logics which satisfies (4) but which does not satisfy (4b).
	\end{theorem}

\begin{proof} For any sets	$I,J$ let $P(I,J)=\{ N_i : i\in I\}\cup\{ D_j : j\in J\}$ and let $\Sig=\{ P(I,J) : I,J\mbox{ are sets}\}$. Let $\Cn$ consist of one unary connective $\lnot$. Let $\gB$ be the algebra with universe $B=\{T,F,D\}$  and with unary operation $\lnot^{\gB}$ interchanging $T$ and $F$ and leaving $D$ fixed. For sets $I,J$, let $\Mng(I,J)$ be the set of all homomorphisms $f:\F(I,J)\to\gB$ such that $f(D_j)=D$ for all $j\in J$. 
For sets $I,J$ we define the logic $\cL(I,J)$ such that its set of atomic formulas is $P(I,J)$, its only connective is $\lnot$ and the class of its meaning functions is $\Mng(I,J)$.  The class of models of $\cL(I,J)$ is defined to be $\Mng(I,J)$ and the validity relation is defined so that $\M\models\phi$ iff $M(\phi)\in\{ T,D\}$. This finishes the definition of the logics $\cL(I,J)$ for all $I,J$.
Intuitively, there are two kinds of atomic formulas in $\cL(I,J)$, the ``normal''  $N_i$s for which the negation $\lnot$ can have its ``normal'' meaning, and the ``dialectic'' $D_j$s whose intended meaning is that their negations are themselves.	
	Let
	\[ \LL = \langle \cL(I,J) : P(I,J)\in\Sig\rangle .\] First we show that $\LL$ is a conditionally substitutional pre-family of logics. Clearly, $\Alg_m(\cL(I,J))=\{\gB\}$ and $\Mng(\cL(I,J))=\{ f\in\Hom(\F(I,J),\gB): S(I,J)\subseteq\ker(f)\}$ where $S(I,J)=\{\langle D_j,\lnot D_j\rangle : j\in J\}$. Thus, $\cL(I,J)$ is a conditionally substitutional logic, for each pair of sets $I,J$, by Claim \ref{claim:subschar}(iii). 
	
	Conditions (1)-(2) in the definition of a family of logics are clearly satisfied by $\LL$. To check (3), notice first that $P(I,J)\subseteq P(I',J')$ exactly when $I\subseteq I'$ and $J\subseteq J'$, and it is clear that $\cL(I,J)$ is the natural restriction of $\cL(I',J')$ in the latter case. For (5), notice first that a bijection $b:P(I,J)\to P(I',J')$ extends to an isomorphism between the corresponding logics if and only if $b$ takes $I$ to $I'$ and $J$ to $J'$. Notice also that $\Sig$ is closed under taking arbitrary unions and under taking subsets. From these, it is easy to check that $\LL$ satisfies (5). Therefore, $\LL$ is a pre-family of logics.
	
	To show that $\LL$ satisfies (4), first we check what $\sim^P$ is. Let $P=P(I,J)$. Let $\omega$ denote the set of nonnegative integers and let $\lnot^kp$ denote $p$ preceded by $k$ copies of $\lnot$, for $k\in\omega$ and $p\in P$. Now, it is easy to check that
	\[ \sim^P = \{\langle\lnot^k N_i,\lnot^{k+2n} N_i\rangle : k,n\in\omega, i\in I\}\cup\{\langle \lnot^k D_j, \lnot^n D_q\rangle : k,n\in\omega, j,q\in J\} .\]
	
	See Figure \ref{example-fig}. Let $h:\F^P\to\gB$. Then $\sim^P\subseteq\ker(h)$ iff $h(D_j)=D$ for all $j\in J$. This shows that Theorem~\ref{gen4}\eqref{4homa} holds, thus $\LL$ satisfies (4). 
	
	To show that $\LL$ does not satisfy (4b), we use Theorem~\ref{thm:4b}\eqref{4bsepa}. Let  $\gA$ be the homomorphic image of $\gB$ via the homomorphism taking $D$ to $a$ and taking $T$ as well as $F$ to $b$ where $a,b$ are distinct. Thus, in $\gA$, both $a$ and $b$ are fixed points of the unary operation. Clearly, $\gA\in\HSP\{\gB\}$ and $\sim^P\subseteq\ker(h)$ for a homomorphism $h:\F^P\to\gA$ iff $h(D_j)=h(D_q)$ for  all $j,q\in J$. Let $P_i=\{D_i\}$ for $i=1,2$, let $Q=\{D_1,D_2\}$  and let $h:\F^Q\to\gA$ be such that $h(D_1)=a$ and $h(D_2)=b$. Then $\sim^{P_i}\subseteq\ker(h)$ for $i=1,2$, but $\sim^Q$ is not a subset of $\ker(h)$ by $a\ne b$, showing that Theorem~\ref{thm:4b}\eqref{4bsepa} fails.	
	\end{proof}

\begin{figure}[ht!]\begin{center}
\begin{tikzpicture}
        \node (D) at (0,1) {$D$};
        \node (F) at (0,0) {$F$};
        \node (T) at (0,-1) {$T$};
        
        \node (A) at (2,1) {$a$};
        \node (B) at (2,-1) {$b$};
        
        \path[arrow] (D) edge (A); 
        \path[arrow] (F) edge (B); 
        \path[arrow] (T) edge (B); 

        \path[arrow] (F) edge[bend right] (T); 
        \path[arrow] (T) edge[bend right] (F); 
        \path[arrow] (D) edge[loop,out=220,in=150,looseness=3] (D); 

        \path[arrow] (A) edge[loop,out=30,in=320,looseness=3] (A); 
        \path[arrow] (B) edge[loop,out=30,in=320,looseness=3] (B); 
  
        \draw [rounded corners=4] (-.7,-1.7) rectangle ++(1.4,3.4);
        \draw (0,2.2) node {$\gB$};
        \draw [rounded corners=4] (-.7+2,-1.7) rectangle ++(1.4,3.4);
        \draw (2,2.2) node {$\gA$};

        \draw [rounded corners=4] (4,0.2) rectangle ++(4.4,1.5);
        \draw (6,2.2) node {$\Fr(\{\gB\}, P(I,J))$};

        \node (Ni) at (5,1.5) {$N_i$};
        \node (N00) at (4.15,1.2) {$\cdot$};
        \node (N01) at (4.25,1.2) {$\cdot$};
        \node (N10) at (4.15,0.3) {$\cdot$};
        \node (N11) at (4.25,0.3) {$\cdot$};
        
        \node (Na1) at (4.5,1.2) {$\cdot$};
        \node (Nb1) at (4.5,0.3) {$\cdot$};
        \path[arrow] (Na1) edge[bend right] (Nb1);
        \path[arrow] (Nb1) edge[bend right] (Na1); 

        \node (Na2) at (4.5+.5,1.2) {$\cdot$};
        \node (Nb2) at (4.5+.5,0.3) {$\cdot$};
        \path[arrow] (Na2) edge[bend right] (Nb2);
        \path[arrow] (Nb2) edge[bend right] (Na2); 

        \node (Na3) at (4.5+.5+.5,1.2) {$\cdot$};
        \node (Nb3) at (4.5+.5+.5,0.3) {$\cdot$};
        \path[arrow] (Na3) edge[bend right] (Nb3);
        \path[arrow] (Nb3) edge[bend right] (Na3);

        \node (N00b) at (4.15+.5+.5+.6,1.2) {$\cdot$};
        \node (N01b) at (4.25+.5+.5+.6,1.2) {$\cdot$};
        \node (N10b) at (4.15+.5+.5+.6,0.3) {$\cdot$};
        \node (N11b) at (4.25+.5+.5+.6,0.3) {$\cdot$};

        \node (Di) at (7,1.5) {$D_j$};
        \node (D00) at (4.15+2.2,1.2) {$\cdot$};
        \node (D01) at (4.25+2.2,1.2) {$\cdot$};
        \node (D10) at (4.15+2.2,0.3) {$\cdot$};
        \node (D11) at (4.25+2.2,0.3) {$\cdot$};
        
        \node (Da1) at (4.5+2.2,1.2) {$\cdot$};
        \node (Db1) at (4.5+2.2,0.3) {$\cdot$};
        \path[arrow] (Da1) edge[bend right] (Db1);
        \path[arrow] (Db1) edge[bend right] (Da1); 

        \node (Da2) at (4.5+.5+2.2,1.2) {$\cdot$};
        \node (Db2) at (4.5+.5+2.2,0.3) {$\cdot$};
        \path[arrow] (Da2) edge[bend right] (Db2);
        \path[arrow] (Db2) edge[bend right] (Da2); 

        \node (Da3) at (4.5+.5+.5+2.2,1.2) {$\cdot$};
        \node (Db3) at (4.5+.5+.5+2.2,0.3) {$\cdot$};
        \path[arrow] (Da3) edge[bend right] (Db3);
        \path[arrow] (Db3) edge[bend right] (Da3);

        \node (D00b) at (4.15+.5+.5+.6+2.2,1.2) {$\cdot$};
        \node (D01b) at (4.25+.5+.5+.6+2.2,1.2) {$\cdot$};
        \node (D10b) at (4.15+.5+.5+.6+2.2,0.3) {$\cdot$};
        \node (D11b) at (4.25+.5+.5+.6+2.2,0.3) {$\cdot$};

        \draw [rounded corners=4] (4,-0.2) rectangle ++(4.4,-1.5);
        \draw (6.3,-2.2) node {$\Fr(\{\gB\}, P(I,J), S(I,J))$};

        \node (FNi) at (5,1.5-1.9) {$N_i$};
        \node (FN00) at (4.15,1.2-1.9) {$\cdot$};
        \node (FN01) at (4.25,1.2-1.9) {$\cdot$};
        \node (FN10) at (4.15,0.3-1.9) {$\cdot$};
        \node (FN11) at (4.25,0.3-1.9) {$\cdot$};
        
        \node (FNa1) at (4.5,1.2-1.9) {$\cdot$};
        \node (FNb1) at (4.5,0.3-1.9) {$\cdot$};
        \path[arrow] (FNa1) edge[bend right] (FNb1);
        \path[arrow] (FNb1) edge[bend right] (FNa1); 

        \node (FNa2) at (4.5+.5,1.2-1.9) {$\cdot$};
        \node (FNb2) at (4.5+.5,0.3-1.9) {$\cdot$};
        \path[arrow] (FNa2) edge[bend right] (FNb2);
        \path[arrow] (FNb2) edge[bend right] (FNa2); 

        \node (FNa3) at (4.5+.5+.5,1.2-1.9) {$\cdot$};
        \node (FNb3) at (4.5+.5+.5,0.3-1.9) {$\cdot$};
        \path[arrow] (FNa3) edge[bend right] (FNb3);
        \path[arrow] (FNb3) edge[bend right] (FNa3);

        \node (FN00b) at (4.15+.5+.5+.6,1.2-1.9) {$\cdot$};
        \node (FN01b) at (4.25+.5+.5+.6,1.2-1.9) {$\cdot$};
        \node (FN10b) at (4.15+.5+.5+.6,0.3-1.9) {$\cdot$};
        \node (FN11b) at (4.25+.5+.5+.6,0.3-1.9) {$\cdot$};

        \node (FDi) at (7,1.5-2) {$D_j=D_q$};
        \node (FDb1) at (4.5+2.2+.4,0.3-1.35) {$\cdot$};
        \path[arrow] (FDb1) edge[loop,out=320,in=240,looseness=4] (FDb1);
    \end{tikzpicture}
    \caption{Illustration of the algebras $\gA$, $\gB$, and the free and conditionally free algebras $\Fr(\{\gB\}, P(I,J))$, and $\Fr(\{\gB\}, P(I,J), S(I,J))$.}
        \label{example-fig}

\end{center}\end{figure}

\subsection{Condition (4) in conditionally substitutional pre\--fa\-mi\-li\-es}\label{subs:cs4}
In this subsection, we discuss (4) in the special case when $\LL$ is conditionally substitutional. 
We show close connections of (4) with the defining relations of the logics and with the patchwork property of models. 
We close with showing existence and non-existence of logic families containing given logics.

The intuition behind condition (4) is expressed quite nicely in a conditionally substitutional pre-family of logics. Recall from Claim \ref{claim:subschar}(iii) that $\Mng_{\cL}$ in a conditionally substitutional logic can be specified by a condition $S$. Theorem \ref{csub4} states that (4) is equivalent to saying that the logic belonging to the union of the signatures is specified by the union of these conditions. 
Note that Theorem \ref{csub4}\eqref{csub4i} is just condition (4).

\begin{theorem}[defining relations and (4)]\label{csub4} 	Let $\LL=\<\cL^P:\;P\in \Sig\>$ be a conditionally substitutional pre-family of logics.
	Let $P_i\in \Sig$ be sets 
	for $i\in I$ such that the union $P=\bigcup_{i\in I}P_i$ belongs to $\Sig$.
	The following conditions are equivalent.
\begin{enumerate}[(i)]\itemsep-2pt
\item\label{csub4i}
$\F^P\slash\!\sim^P\  =\	\Fr(\Alg_m(\cL^P),\; P,\; \cup\!\sim^{P_i})$.
\item \label{4si}
$\F^P\slash\!\sim^P\  =\  \Fr(\Alg_m(\cL^P),\; P,\; \cup S_i)$ whenever $\F^{P_i}\slash\!\sim^{P_i} = \Fr(\Alg_m(\cL^{P_i}),\; P_i,\; S_i)$ for all $i\in I$.
\end{enumerate}
\end{theorem}

\begin{proof}    \def\ACSA{\quad\Leftrightarrow\quad}
    We start with the general observation that for 
    any class $\K$ of similar algebras, set $P$ and
    binary relations $R,S\subseteq F^P\times F^P$, 
    \begin{align*}
        \Fr(\K, P, R) = \Fr(\K, P, S) &\ACSA
        \Cr(\K,P,R)=\Cr(\K,P,S) \\ &\ACSA
        \Hom(\K,P,R)=\Hom(\K,P,S) \\ &\ACSA 
         \big(\forall h\in \Hom(\F^P,\K)\big)\big(
                R\subseteq\ker(h) \ \Leftrightarrow\  S\subseteq\ker(h)\big)
    \end{align*}
    Consider now the following three statements.
    \begin{enumerate}[(A)]\itemsep-2pt
        \item $(\forall h:\F^P\to\Alg_m(\cL^P))\ 
                \big(\sim^{P}\subseteq\ker(h) 
                \ \Leftrightarrow \ 
                \cup\sim^{P_i}\subseteq\ker(h)\big)$.
        \item $(\forall h:\F^{P_i}\to\Alg_m(\cL^{P_i}))\ 
                \big(\sim^{P_i}\subseteq\ker(h) 
                \ \Leftrightarrow \ 
                S_i\subseteq\ker(h)\big)$\ \  for all $i\in I$.
        \item $(\forall h:\F^P\to\Alg_m(\cL^P))\ 
                \big(\sim^{P}\subseteq\ker(h) 
                \ \Leftrightarrow \ 
                \cup S_i\subseteq\ker(h)\big)$.                
    \end{enumerate}
    Then (i) is equivalent to (A); and (ii) is equivalent to (B)$\Rightarrow$(C). Take any $h:\F^P\to\Alg_m(\cL^P)$, and
    for $i\in I$, write $h_i = h\upharpoonright F^{P_i}:\F^{P_i}\to\Alg_m(\cL^{P_i})$. \\

    Proof of (i)$\Rightarrow$(ii): Suppose (A) and (B) hold. Then
    \begin{align*}
        \sim^P\subseteq\ker(h) 
            \quad&\overset{(A)}{\Longleftrightarrow}\quad
        \cup\sim^{P_i}\subseteq\ker(h)  
            \quad\overset{}{\Longleftrightarrow}\quad
        (\forall i\in I)\ \sim^{P_i}\subseteq \ker(h_i) \\
            &\overset{(B)}{\Longleftrightarrow}\quad
        (\forall i\in I)\ S_i\subseteq \ker(h_i) 
            \quad\overset{}{\Longleftrightarrow}\quad
        \cup S_i \subseteq  \ker(h),
    \end{align*}
    therefore (C) holds. In this direction we did not use 
    that $\LL$ is conditionally substitutional. \\

    Proof of (ii)$\Rightarrow$(i): By Claim \ref{claim:subschar}(iii),
    for each $P_i$, $\Mng^{P_i}$ can be specified by a condition $S_i$, 
    that is, $\Mng^{P_i} = \Hom(\F^{P_i},\Alg_m(\cL^{P_i}), S_i)$. By the
    same claim, $\Mng^{P_i} = \Hom(\F^{P_i},\Alg_m(\cL^{P_i}), \sim^{P_i})$,
    consequently, (B) holds for the $S_i$'s. As we assumed (ii) and (B) holds
    it follows that (C) holds as well. Then
        \begin{align*}
        \sim^P\subseteq\ker(h) 
            \quad&\overset{(C)}{\Longleftrightarrow}\quad
        \cup S_i\subseteq\ker(h)  
            \quad\overset{}{\Longleftrightarrow}\quad
        (\forall i\in I)\ S_i\subseteq \ker(h_i) \\
            &\overset{(B)}{\Longleftrightarrow}\quad
        (\forall i\in I)\ \sim^{P_i}\subseteq \ker(h_i) 
            \quad\overset{}{\Longleftrightarrow}\quad
        \cup \sim^{P_i} \subseteq  \ker(h),
    \end{align*}
    therefore (A) holds.
\end{proof}

We turn to connections with the quite intuitive patchwork property of models, see, e.g., \cite[Def. 4.3.14]{AGyNS2022}.

\begin{definition}[patchwork property]\label{def:patchwork}
	Let $\LL = \<\cL^P:\;P\in \Sig\>$ be a class of logics. $\LL$ has the 
	\emph{patchwork property of models} if the following holds.
	Let $P_i\in \Sig$ for $i\in I$ be such that $P=\cup_iP_i\in \Sig$.
	Let $\gM_i\in M^{P_i}$ be models such that for any $i,j\in I$
	whenever $F^{P_i\cap P_j}\neq \emptyset$, then
	\begin{equation}
		\mng_{\gM_i}^{P_i}\upharpoonright F^{P_i\cap P_j} = 
		\mng_{\gM_j}^{P_j}\upharpoonright F^{P_i\cap P_j}. \label{eq:azonosak}
	\end{equation}
	Then there is $\gM\in M^{P}$ such that
	\begin{equation}
		\mng_{\gM}^P \upharpoonright F^{P_i}\;=\;\mng_{\gM_i}^{P_i} \label{eq:azonos}
	\end{equation}
	for all $i\in I$.\endef
\end{definition}

The next theorem says that in conditionally substitutional pre-families of logics, the patchwork property implies (4) 
but the patchwork property and (4b) remain independent of each other. These facts suggest that it may be better to have (4) in the definition of a logic family than (4b).

\begin{theorem}[the patchwork property and (4)]\label{thm:pwp2}
	Suppose $\LL = \<\cL^P:\;P\in \Sig\>$ is a conditionally substitutional pre-family of logics. Then (i)-\eqref{amalg4b}  below hold.
	\begin{enumerate}[(i)]\itemsep-2pt
	\item	$\LL$ satisfies (4) if it has the patchwork property of models. 
	\item  (4b) is not implied by the patchwork property. 
	\item \label{amalg4b} (4b) does not imply the patchwork property.
	\end{enumerate}
\end{theorem}
\begin{proof} Proof of (i): 
	Assume the patchwork property, we show (4) by checking the equivalent condition Theorem \ref{gen4}\eqref{4homaa}. Take signatures
	$P_i\in \Sig$ for $i\in I$ such that $P = \cup_iP_i\in \Sig$. Take any homomorphism
	$h:\F^P\to\gA$ for some $\gA\in\Alg_m(\cL^P)$ such 
	that $\sim^{P_i}\subseteq\ker(h)$ holds for $i\in I$.
	By the conditional substitution property, the homomorphisms $h\upharpoonright F^{P_i}$
	are meaning homomorphisms $\mng_{\gM_i}^{P_i}$ for some models $\gM_i\in M^{P_i}$.
	Clearly, \eqref{eq:azonosak} holds for $i,j\in I$. Therefore, by the 
	patchwork property of models, there is $\gM\in M^P$ that fulfills 
	\eqref{eq:azonos}. But then $h=\mng_{\gM}^P$, implying $\sim^P\subseteq\ker(h)$.
	
	Proof of (ii): The conditionally substitutional pre-family of
    logics $\LL=\<\cL(I,J): P(I,J)\in\Sig\>$ from Theorem \ref{example}
    does not satisfy (4b) but has the patchwork property of models.

    Proof of (iii): As substitutionality implies (4b) by Corollary \ref{cor:subs}, it is enough to construct a substitutional 
    pre-family of logics $\LL$ that lacks the patchwork property. 

    Let $\K$ be a class of similar algebras that does not have
    the amalgamation property.\footnote{For the definition of the amalgamation property, see e.g. \cite[Def. 4.4.17]{AGyNS2022}. } 
	For example, let $\gB_1 = \<\{0\}, \cap\>$ and $\gB_2=\<\{0,1\},\cap\>$ be the 
	$\cap$-reducts of
	the one and two element Boolean set algebras respectively, and take $K=\{\gB_1, \gB_2\}$.
	The V-formation in Figure \ref{fig:V} cannot be completed to an amalgamation diagram in $K$, 
	showing that $K$ fails to have the amalgamation property.
	\begin{figure}[ht!]
	\begin{center}
	\begin{tikzpicture}[-,>=stealth',shorten >=1pt,auto,node distance=2cm,
	                main node/.style={}]

		\node (A) at (0,0) {};
		\draw [fill] (0,0) circle [radius=.05];
		\draw (0,0) circle [radius=.09];
		
		\node (B1) at (1,1-0.2) {};
		\draw (1,0.7-0.3) node {} -- (1,1.3-0.3) node {};
		\draw [fill] (1,0.7-0.3) circle [radius=.05];
		\draw [fill] (1,1.3-0.3) circle [radius=.05];
		\draw (1,1.3-0.3) circle [radius=.09];
		
		\node (B2) at (1,-1+0.2) {};
		\draw (1,-0.7+0.3) node {} -- (1,-1.3+0.3) node {};
		\draw [fill] (1,-0.7+0.3) circle [radius=.05];
		\draw [fill] (1,-1.3+0.3) circle [radius=.05];
		\draw (1,-1.3+0.3) circle [radius=.09];
		\path (A) edge[->] (B1)
			  (A) edge[->] (B2);
		\node (At) at  (-0.5,0) {$\gB_1$};
		\node (B1t) at (1.5,1-0.3) {$\gB_2$};
		\node (B2t) at (1.5,-1+0.3) {$\gB_2$};
	\end{tikzpicture}
	\end{center}
	\caption{The embeddings of $\gB_1$ into $\gB_2$ showing the failure of the 
	amalgamation property in $K$.}
	\label{fig:V}
	\end{figure}
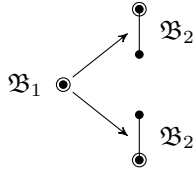
	
	Let $\Cn$ be the similarity type
    of the algebras in $\K$, and $\Sig$ be the class of all sets
    $P$ not containing proper $\Cn$-type terms. For $P\in\Sig$
    we let $\Mng^P=\Hom(\F^P,\K)$. The class of models $M^P$
    is defined to be $\Mng^P$, and the validity relation is defined 
    so that $\gM\models\varphi$ always holds. This way we defined
    the logic $\cL^P$. Let $\LL = \<\cL^P:\; P\in\Sig\>$.
    That $\LL$ is a pre-family of logics is straightforward. By Claim
    \ref{subsequi}, $\LL$ is substitutional, and thus it is conditionally 
    substitutional, as well. It remained to check that $\LL$ does not
    have the patchwork property. For, take $\gA_0, \gA_1, \gA_3\in \K$
    such that $\gA_0\subseteq \gA_1$, $\gA_0\subseteq\gA_2$, but
    there is no $\gA\in\K$ and embeddings $e_i:\gA_i\to\gA$ ($i=1,2$)
    with $e_1(\gA_0)=e_2(\gA_0)$. Such algebras exist because $\K$
    lacks the amalgamation property. Now, $A_1, A_2\in\Sig$ 
    and $A_1\cup A_2\in\Sig$ as well. For $i=1,2$, let 
    $m_i:\F^{A_i}\to \gA_i$ be the homomorphic extension of the identity mapping $A_i\to A_i$. Each $m_i$ is a meaning function, and
	\begin{align}
		m_1\upharpoonright F^{A_1\cap A_2} = 
        m_2\upharpoonright F^{A_1\cap A_2},
	\end{align}
    because $A_1\cap A_2=A_0$. On the other hand, there cannot
    exist $m:\F^{A_1\cup A_2}\to\K$ with
    \[
        m\upharpoonright \F^{A_1} = m_1, \text{ and }
        m\upharpoonright \F^{A_2} = m_2,
    \]
    because in this case $\gA = m(\F^{A_1\cup A_2})$ would be
    an amalgam in $\K$.
\end{proof}


The next theorem says that in conditionally substitutional families of logic, the tautological formula algebras determine the logics themselves, up to meaning-isomorphisms. 
The proof of this theorem uses condition (4) quite inherently. 

\begin{theorem}[renaming atomic formulas and (4)]\label{3-theorem}
	Assume that $\LL=\langle\cL(P) : P\in\Sig\rangle$ is a conditionally substitutional family of logics. Then $\LL$ satisfies (6a), i.e., any bijection $b:P\to Q$ that extends to a bijection between the tautological formula algebras of $\cL^P$ and $\cL^Q$ induces a meaning-isomorphism between $\cL^P$ and $\cL^Q$, whenever $P,Q\in\Sig$.
\end{theorem}

\begin{proof} Assume that $b:P\to Q$ is a bijection between $P,Q\in\Sig$ that extends to an isomorphism $k$ between the tautological formula algebras. Let $P',Q'\in\Sig$ be such that $\cL(P')$ and $\cL(Q')$ are isomorphic copies of $\cL(P)$ and $\cL(Q)$ respectively, say, via the isomorphisms $i$ and $j$ and such that $R=P'\cup Q'\in \Sig$ and $P'$ and $Q'$ are disjoint. There are such $P',Q'$ by condition (5)b. Assume that $f\in\Mng^P$, we will show that $g=f\circ{(\overline{b})}^{-1}\in\Mng^Q$, where $\overline{b}:\F^P\to\F^Q$ is the extension of $b$ to $\F^P$. 		
		
Let $f'=f\circ i$, then $f'\in\Mng^{P'}$ since $i$ is an isomorphism between $\cL(P)$ and $\cL(P')$.  Thus 
$\sim_{P'}\subseteq\ker(f')$.
By $P'\subseteq R$ and condition (3), $f'$ is the restriction of some $f''\in\Mng^R$. Let $f'':\F^R\to\gA\in\Alg_m(\cL^R)$.

Let now turn to $g$. Since $\sim^P\in\ker(f)$ by $f\in\Mng^P$, and $b$ extends to an isomorphism between the tautological formula algebras, we have that $\sim^Q\subseteq\ker(g)$. For details, see Figure \ref{6-fig}. Let $g'=g\circ j$, then $\sim^{Q'}\subseteq\ker(g')$ since $j$ is an isomorphism, and $g':\F^{Q'}\to\gA$ because the ranges of $f$ and $g$ coincide.

Define $h$ on $R$ such that $h$ is $f''$ on $P'$ and $h$ is $g'$ on $Q'$. Let us denote also by $h$ the extension to $\F^R$. Then $h:\F^R\to\gA\in\Alg_m(\cL^R)$ such that $\ker(h)\supseteq\;\sim^{P'}\cup\sim^{Q'}$. Then $\ker(h)\supseteq\;\sim^R$ by condition (4) and Theorem~\ref{gen4}\eqref{4homa}. 
	Thus, $h\in\Mng^R$ since $\cL(R)$ is conditionally substitutional. Thus the restriction of $h$ to $F(Q')$ is in $\Mng^{Q'}$, so $g\in\Mng^Q$ since $j$ is an isomorphism between $\cL(Q)$ and $\cL(Q')$. We have seen that $\Mng^P\subseteq\{ g\circ \overline{b} : g\in\Mng^Q\}$. Now,	
$\Mng^Q\subseteq\{ g\circ (\overline{b}^{-1}) : g\in\Mng^P\}$ can be proved the same way, so $\Mng^P=\{ g\circ \overline{b} : g\in\Mng^Q\}$.	
\end{proof}

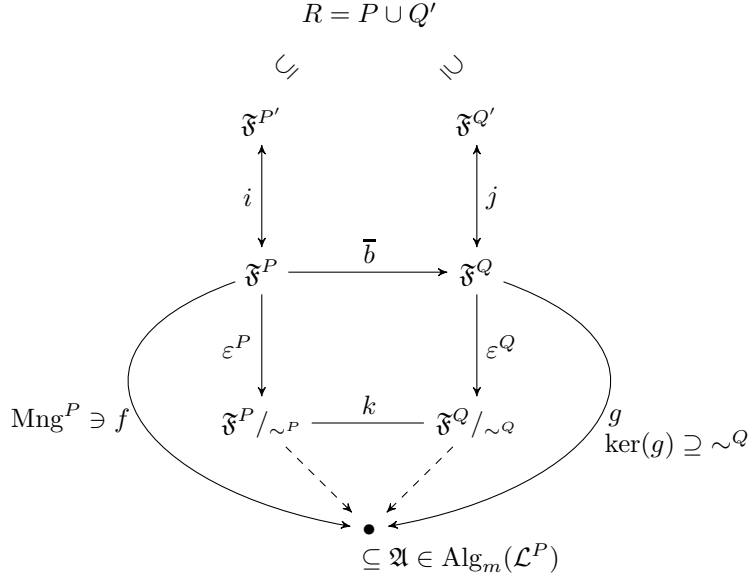
\begin{figure}[ht!]\begin{center}
\begin{tikzpicture}[->,>=stealth',shorten >=1pt,auto,node distance=2cm,
                main node/.style={}]

  \node[main node] (O) {$R=P\cup Q'$};
  \node[main node] (A1) [below left of=O] {$\F^{P'}$};
  \node[main node] (B1) [below right of=O] {$\F^{Q'}$};

  \node[main node] (A2) [below of=A1] {$\F^{P}$};
  \node[main node] (B2) [below of=B1] {$\F^{Q}$};

  \node[main node] (A3) [below of=A2] {$\F^{P}/_{\sim^{P}}$};
  \node[main node] (B3) [below of=B2] {$\F^{Q}/_{\sim^{Q}}$};

  \node[main node] (D) [below right of=A3] {$\bullet$};

  \path 
    (A1) edge[<->] node[left] {$i$} (A2) 
    (B1) edge[<->] node[right] {$j$} (B2) 
    (A2) edge[->] node[left] {$\varepsilon^P$} (A3) 
    (B2) edge[->] node[right] {$\varepsilon^Q$} (B3)
    (A2) edge[->] node[above] {$\overline{b}$} (B2) 
    (A3) edge[-] node[above] {$k$} (B3)
    (A3) edge[dashed] node {} (D)
    (B3) edge[dashed] node {} (D); 
  \path  
    (A2) edge[out=200,in=170,looseness=2] node[left] 
        {$\Mng^P\ni f$} (D);
  \path
    (B2) edge[looseness=2,out=-20,in=10] node[right] {$g$} (D);
  
  \node (X) at (1.2,-7.2) {$\subseteq\gA\in\Alg_m(\cL^P)$};
  \node (Y) at (4.05,-5.7) {$\ker(g)\supseteq\;\sim^{Q}$};

  \node (Z) at (-1.1,-0.7) {\rotatebox[origin=c]{45}{$\subseteq$}};
  \node (W) at (1.1,-0.7) {\rotatebox[origin=c]{-45}{$\supseteq$}};

\end{tikzpicture}
\caption{Illustration for the proof of Theorem \ref{3-theorem}.}
\label{6-fig}
\end{center}\end{figure}

\bigskip

We close the section by investigating how restrictive the notion of logic families is.

	\begin{proposition}[existence of logic families]\label{thm:restr}   The following statements are true.
		\begin{enumerate}[(i)]\itemsep-2pt
			\item\label{existsubs} Each substitutional meaning-determined logic is a member of a 
			substitutional logic family with class of signatures being all sets not containing proper terms by the connectives.
			\item\label{existcsubs}  Each conditionally substitutional logic is a member of some conditionally substitutional logic family.
			\item\label{nonexist} There is a logic which is not a member of any logic family, but each logic is a member of a pre-family of logics.
	\end{enumerate}\end{proposition}
	
	\begin{proof} Proof of \eqref{existsubs}: Let $\cL$ be a substitutional meaning-determined logic with $\Cn$ as connectives. Let $\Sig$ be the class of all sets not containing proper terms written up by $\Cn$. For all $P\in\Sig$ let $\cL^P$ be ``the substitutional meaning-determined logic'' with $P$ as atomic formulas, $\Cn$ as connectives, with $\Mng(\cL^P)=\Hom(\F^P,\Alg_m(\cL))$ and with validity inherited from $\cL$. It is easy to check that $\LL$ is a substitutional pre-family of logics, by using that $\cL^P$ is an isomorphic copy of $\cL^Q$ iff $P$ and $Q$ have the same cardinality. Then $\LL$ is a family of logics by Corollary \ref{cor:subs}. Clearly, $\cL$ is a member of this family.
		
		Proof of \eqref{existcsubs}: Let $\cL=\langle F, M, \mng, \models\rangle$ be a conditionally substitutional logic.
		The idea is that we make $\LL$ to consist of all unions of disjoint (almost) isomorphic copies of $\cL$. Let $\Sig = \{ P\times I : I\mbox{ is a set and }P\cap\Cn=\emptyset\}$. 
		Assume $Q=P\times I\in\Sig$. Let us define $\cL^Q=\langle F^Q, M^Q , \mng^Q, \models^Q\rangle$ such that its set of atomic formulas is $Q$, its set of connectives is that of $\cL$, and meaning and validity are defined as follows.  The models are those $I$-sequences of models of $\cL$ whose meaning functions go to a common $\gA\in\Alg_m(\cL)$, i.e., $M^Q = \{\langle\gM,\gA\rangle\in{}^IM\times\Alg_m(\cL) : (\forall i\in I)\mng_{\gM_i}(\F)\subseteq \gA\}$. For an $\langle\gM,\gA\rangle\in M^Q$ the meaning function is then $h:\F^Q\to\gA$ such that $h(\< p,i\>)=\mng(\gM_i,p)$ for all $\< p,i\>\in Q$.
		Validity is defined as $\langle\gM,\gA\rangle\models\phi$ if and only if $\phi\in F^{(P\times\{ i\})}$ for some	$i\in I$ and $\gM_i\models\phi'$ where $\phi'$ is $\phi$ such that we replace $p\in P$ with $\< p,i\>$ in it, for each $p\in P$. This defines $\LL$. Conditions (1),(2) are clearly satisfied and it is not difficult to check that (3) is also satisfied. Each bijection $b:I\to J$ induces an  isomorphism between $\cL^{P\times I}$ and $\cL^{P\times J}$, so condition (5) is satisfied. Hence, $\LL$ is a pre-family of logics. It is not difficult to check that $\Mng(\cL^{P\times I})=\Hom(\F^{P\times I}, 
		\Alg_m(\cL), S(I))$ where $S(I)=\bigcup\{\ h_i(\sim_{\cL}) : i\in I\}$ where $h_i:\F^P\to\F^{P\times\{ i\}}$ is the isomorphism induced by $h(p)=\< p,i\>$ for all $p\in P$. Here, one uses that $\cL$ is conditionally substitutional together with Claim \ref{claim:subschar}. Thus, $\LL$			
		 is a conditionally substitutional pre-family of logics. It satisfies condition (4) by Theorem \ref{csub4}\eqref{4si}. Though, strictly speaking, $\cL$ is not a member of this family, it is easy to change $\LL$ so that this will hold. Namely, $\cL^{P\times\{ i\}}$ may not be isomorphic to $\cL$ because  the models of the former are $\{\langle\gM,\gA\rangle\in M\times\Alg_m(\cL) : \Rng(\mng_{\gM})\subseteq\gA\}$, thus each model of $\cL$ may appear in $\cL^{P\times\{ i\}}$ in several copies. However, $P\times\{ i\}$ are minimal elements of $\Sig$, so we can replace their classes of models with $M$ itself. 		 
		
		Proof of \eqref{nonexist}: Let $\cL=\langle F, M, \mng, \models\rangle$ be such that $|P|\ge 2$ and $\Alg_m(\cL)=\{\gA\}$ where $\gA$ is the $\Cn$-free algebra freely generated by $a\in A$ and $\mng(\gM,p)=a$ for all $p\in P$. Assume that $\cL=\LL(P)$ for some logic family $\LL$. We will derive a contradiction. There is $Q\in\Sig$ disjoint from $P$ such that $\cL^Q$ is an  isomorphic copy of $\cL$, by (5)a. There are $P',Q'\in\Sig$ such that $R=P'\cup Q'\in\Sig$ and $\cL^{P'}$ and $\cL^{Q'}$ are isomorphic copies of $\cL^P$ and $\cL^Q$ respectively, by (5)b. Then $\mng^{P'}(\gM,p)=\mng^{Q'}(\gN,q)=a$ for all models $\gM, \gN$ of $\cL^{P'}, \cL^{Q'}$ respectively, and for all $p\in P'$ and $q\in Q'$, since $\cL^{P'}$ and $\cL^{Q'}$ are isomorphic copies of $\cL$. 
		By (3), then the same holds for $\cL^{R}$, i.e., $\mng^R(\gM,p)=a$ for all models $\gM$ of $\cL^R$ and $p\in R=P'\cup Q'$. This means that $P'\times Q'\subseteq\sim^R$.	Notice now that only ``trivial" quasi-equations are true in $\gA$ because it is a free algebra, therefore
		$P'\times Q'$ is not a subset of the $\vd_{\cL^R}$-closure of $P'\times P'\cup Q'\times Q'$.	
		Therefore, Theorem \ref{gen4}\eqref{4subd} does not hold, i.e., $\LL$ cannot satisfy (4), a contradiction.	
	\end{proof}

\section{The key examples as logic families}\label{sec:exa}

In this section, we show that both of our key examples, classical propositional logic and first-order
logic, fit in the scheme of logic family. By this we mean that their definitions suggest natural systems of logics, and we show that these systems are indeed logic families.  We also show that the systems formed of the finite-variable fragments of the first-order logic are not all logic families, although they have the  patchwork property of models. 
\medskip

Let $\CPL(P)$ denote classical propositional logic with atomic formulas $P$, see Example \ref{def:proplogic}.

\begin{theorem}\label{thm:proplogicgeneral} 
	For classical propositional logic, the class
	\begin{equation}
		\LL_C = \<\CPL(P)\; :\; P\text{ is a set not containing proper terms by }\{\land,\lnot,\bot\}\>
	\end{equation}
	is a logic family in the sense of Def.\ref{def:generallogic}.
\end{theorem}
\begin{proof}
	Let $\Sig$ be the class of all nonempty sets $P$ which do not contain proper terms written up by $\{\land,\lnot,\bot\}$.
	Items (1), (2) of \ref{def:generallogic} are straightforward. 
	For (3), it is easy to check that $\CPL(P)$ is a reduct of $\CPL(Q)$ whenever $P\subseteq Q$.

	As for (5), any bijection $f:P\to Q$ for sets $P$ and $Q$ induces an isomorphism
	between the corresponding logics $\CPL(P)$ and $\CPL(Q)$: $f^F$ is the homomorphic 
	extension of $f$ to $f^F:\F^P\to\F^Q$, and for a model $\gM:P\to 2$
	let $f^M$ assign the model $f^M(\gM):Q\to 2$ defined by $f^M(\gM)(x) = \gM(f^{-1}(x))$. 
	From this, (5) follows.
	
	(4) follows now from Corollary \ref{cor:subs} since $\CPL$ is substitutional by \ref{ex:pcsubst}.

\end{proof}

Let us turn to our main example. 
For a fixed nonempty set $V$ of variables, $\LL_{FOL}(V)$ is defined as follows. 
$\Sig(V)$ contains the sets $P_t(V)$ for similarity types $t$ (see Example \ref{def:FOL}).
Now, each $P_t(V)$ determines the similarity type $t$ and thus the logic
$\FOL_t(V)$. It follows that the class
\begin{eqnarray}
	\LL_{FOL}(V) &=& \< \FOL_t(V):\; P_t(V)\in \Sig(V) \> 
\end{eqnarray}
is well defined. 

In the case $V=\emptyset$, for any similarity type $t$ we have 
$P_t(V) =\emptyset$ but $F_t(V)$ is not empty as there is a constant symbol 
$\bot$ among the connectives. 
Set $\Sig(\emptyset) = \{\emptyset\}$ and $\LL_{FOL}(\emptyset)=\< \FOL_\emptyset(\emptyset):\; P_\emptyset(\emptyset)\in \Sig(\emptyset) \>$. 
Then $\LL_{FOL}(\emptyset)$
is a one-element sequence. Recall that \ref{def:generallogic}(5)c prevents us to
call $\LL_{FOL}(\emptyset)$ a logic family (but this is only a convention that comes handy in certain situations).

\begin{theorem}\label{thm:pwp1}
	Propositional logic $\LL_C$ and first-order logic $\LL_{FOL}(V)$, 
	for any set $V$ of variables, have the patchwork property of models.
\end{theorem}
\begin{proof}
	As for $\LL_C$, take sets of propositional variables $P_i$ and let
	$P=\cup_iP_i$. Let $\gM_i\in M^{P_i}$ be such that \eqref{eq:azonosak} holds.
	Each $\gM_i$ is a function $\gM_i:P_i\to \{0,1\}$, and \eqref{eq:azonosak}
	is equivalent to requiring 
	\begin{equation}
		\gM_i\upharpoonright P_i\cap P_j \ = \ \gM_j\upharpoonright P_i\cap P_j\,,
	\end{equation}
	whenever $P_i\cap P_j\neq\emptyset$. But in this case the union
	$\gM = \bigcup_i\gM_i$ is a well-defined function $\gM:P\to \{0,1\}$, and 
	clearly \eqref{eq:azonos} holds. 
	
	Let us turn to the $\LL_{FOL}(V)$ case. If $V=\emptyset$, then there
	is nothing to prove as the patchwork property holds vacuously. Otherwise,
	models of $\LL_{FOL}(V)$ are determined by the similarity types $t$
	independently of the set $V$. Fix the nonempty set $V$.
	
	Take $P_i\in \Sig$ for $i\in I$ such that $P=\cup_iP_i\in \Sig$. Then there
	are similarity types $t_i$ such that $P_i = P_{t_i}(V)$ for $i\in I$, 
	and $t=\cup_i t_i$ is a similarity type such that $P=P_t(V)$. 
	If the $t_i$-type first-order models 
	$\gM_i = \<M_i, r^{\gM_i}\>_{r\in\dom(t_i)}$ satisfy \eqref{eq:azonosak},
	then for each $i, j\in I$, the interpretations of the symbols belonging
	to both $t_i$ and $t_j$ are the same in $\gM_i$ and $\gM_j$. 
	Also, the constant $\top$ (truth) has 
	the same interpretation in every $\gM_i$, meaning that these models have
	the same universe $M=M_i$ ($i\in I$). 
	It follows that the model $\gM = \<M, r^{\gM_i}\>_{r\in \dom(t_i), i\in I}$
	is well defined and the corresponding meaning function satisfies 
	\eqref{eq:azonos}.
\end{proof}

\begin{theorem}\label{thm:FOLgeneral}
	For an infinite set $V$ of variables, the class
	\begin{equation}
		\LL_{FOL}(V) = \< \FOL_t(V):\; P_t(V)\in\Sig(V)\>
	\end{equation}
	is a logic family in the sense of Def.\ref{def:generallogic}.
\end{theorem}
\begin{proof}
	Let $V$ be a nonempty set of variables. 
	Verifying items (1), (2), (5) of \ref{def:generallogic} 
	are straightforward. 
	As for \ref{def:generallogic}(3): 
	if $P\subseteq Q$ and $P, Q\in \Sig(V)$, then the similarity
	type $t_Q$ corresponding to $Q$ is an extension of that of $P$ 
	(i.e., $t_P\subseteq t_Q$). Let us define the reduct function $f:M^Q\to M^P$ such that $f$ assigns to a $t_Q$-type model its $t_P$-type reduct which is
	a $t_P$-type model. Every $t_P$-type model can be extended into a 
	$t_Q$-type one and so
	\begin{equation}
		\{\mng^P_{\gM} : \gM\in M^P\} = 
		\{ \mng^Q_{\gM}\upharpoonright F^P:\; \gM\in M^Q\},
	\end{equation}
	i.e.,  $f$ is surjective and meanings are respected by it. Now, referential transparency of first-order logic provides that also validity is respected by $f$. Hence, $\cL^P$ is a reduct of $\cL^Q$.
	Thus, $\LL_{FOL}(V)$ is a pre-family of logics, for all nonempty $V$.
	
	To check (4), assume that $V$ is infinite. Then $\LL_{FOL}(V)$ is 
	conditionally substitutional by Theorem \ref{thm:FOLconsubs}, so (4) holds by Theorem \ref{thm:pwp2}(i)
	since $\LL_{FOL}(V)$ has the patchwork property of models by Theorem \ref{thm:pwp1}.
	\end{proof}

\begin{theorem}\label{thm:nofamily}
	If $V$ is finite and $|V|\ge 2$, then $\LL_{FOL}(V)$ is not a logic family but it is a pre-family of logics.
\end{theorem}
\begin{proof}
We have already seen in the previous proof that $\LL_{FOL}(V)$ is a pre-family of logics. Assume that $V$ is finite and $|V|\ge 2$, we have to show that $\LL_{FOL}(V)$ does not satisfy condition (4). 

Assume first that $|V|=2$, say $V=\{ v_0,v_1\}$. Let the similarity type $t$ have three binary relation symbols $R_1,R_2,R_3$ and let $\cL$ denote $\FOL_t(V)$. Define the model $\gM\in M_t$ as $\gM=\langle M,R_1^{\gM},R_2^{\gM},R_3^{\gM}\rangle$ with  \[ M=\{ 0,1,2,3,4,5,6\}\ \mbox{ and }\ R_i^{\gM}=\{ \< j,j+i\> : j\in M\}\cup \{ \< j,j-i\> : j\in M\} \]   where $+$ and $-$ are understood as addition and subtraction modulo $7$ and $i=1,2,3$. Let $\gA=\mng_{\gM}(\F_t(V))$, the concept algebra of $\gM$. Define $h:\F_t(V)\to\gA\in\Alg(\cL)$ by
\[ h(R_i(v_0,v_1))=R_i^{\gM},\quad h(R_i(v_1,v_0))=R_j^{\gM}\cup R_k^{\gM},\quad \mbox{ and } \]
 $h(R_i(v_0,v_0))=h(R_i(v_1,v_1))=\emptyset$ for $\{ i,j,k\}=\{ 1,2,3\}$.
 
First we show that $\sim_{\cL}\subseteq\ker(h)$ does not hold.
Indeed, the pair
\[ \langle\ \exists v_0\exists v_1(R_1(v_0,v_1)\land R_2(v_0,v_1)),\ \exists v_0\exists v_1(R_1(v_1,v_0)\land R_2(v_1,v_0))\ \rangle \]
is in the tautological congruence of $\cL$, but it is not in $\ker(h)$. (We note that this pair expresses that a relation is empty iff its converse is empty.)

Let $t=t_1\cup t_2\cup t_3$ where $t_i$ is the similarity type with one binary relation symbol $R_i$ and let $\sim_i$ be the tautological congruence relation of $\cL_i=\FOL_{t_i}(V)$, for $i=1,2,3$. We are going to show that $\sim_i\subseteq\ker(h)$ for $i=1,2,3$, this will show that (4) fails for $\cL$ by Theorem \ref{gen4}\eqref{4homa}. 

To show  $\sim_1\subseteq\ker(h)$, we will show that the subalgebra  $\gB$ of $\gA$ generated by $\{ h(R_1(x,y)) : x,y\in V\}$ is isomorphic to the meaning algebra of a model of $t_1$. Clearly, $\gA$ is a finite Boolean set algebra with some extra operations, so $\gB$ is finite, and it is easy to see by induction, that $\Id_M=h(v_0=v_1)$, ${R_1}^{\gM}$, and ${R_2}^{\gM}\cup {R_3}^{\gM}$ are atoms of $\gB$. Since their union is the unit of the set algebra, they are all the atoms of $\gB$. Each of these is a binary relation with domain and range $M$, thus the operations $\exists v_j^{\gB}$ for $j=1,2$ take all nonempty elements of $\gB$ to $M\times M$ and take $\emptyset$ to itself. Let now $\gN=\langle N, R_1^{\gN}\rangle$ where
\[ N=\{0,1,2\}\quad\mbox{ and }\quad R_1^{\gN}=\{\<0,1\>, \<1,2\>, \<2,0\>\} .\]
Let $\gC$ denote the meaning algebra of $\gN$. It is easy to see that $\gC$  also has three atoms, $\Id_N=\mng_{\gN}(v_0=v_1)$, $\mng_{\gN}(R_1(v_0,v_1))$ and $\mng_{\gN}(R_1(v_1,v_0))$, with all these having domain and range $N$. Therefore, $\{\< h(\phi),\mng_{\gN}(\phi)\> : \phi\in F_{t_1}(V)\}$ is an isomorphism between $\gB$ and $\gC$. This means that $\sim_1\subseteq\ker(\mng_{\gN})\subseteq\ker(h)$ and we are done.
The proofs for $\sim_i\subseteq\ker(h)$, $i=2,3$ are completely analogous. We have seen that (4) does not hold for $\LL_{FOL}(V)$ if $|V|=2$.

Assume now $|V|=n\ge 3$, say $V=\{ v_i : 0\le i< n\}$. The proof is analogous but a bit more complex because we have more variable-substitutions.  Let the similarity type $t$ have three $n$-place relation symbols $R_1,R_2,R_3$. 
Let $U_0=U_1=\{0,1,\dots,n+6\}$, for $2\le k<n$ let $U_k=\{(k-1)(n+7), (k-1)(n+7)+1, \dots, (k-1)(n+7)+n+6\}$ and let $M=\bigcup\{ U_k : 0\le k<n\}$. Let $W = \{ s\in{}^nM : (\forall k<n)s_k\in U_k\}$. Define the model $\gM=\langle M, R_1^{\gM}, R_2^{\gM}, R_3^{\gM}\rangle$ as
\begin{description}
	\item{} $R_1^{\gM} = \{ s\in W : |s_0-s_1|=1\ \mbox{ iff }\ \sum\{ s_k : 2\le k<n\}=1\}$,
	\item{} $R_2^{\gM} = \{ s\in W : |s_0-s_1|=2\ \mbox{ iff }\ \sum\{ s_k : 2\le k<n\}=2\}$,
	\item{} $R_3^{\gM} = \{ s\in W : |s_0-s_1|\ge 3\ \mbox{ iff }\ \sum\{ s_k : 2\le k<n\}\ge 3\}$,
\end{description}
where subtraction and sum are meant modulo $n+7$. Let $\gA=\mng_{\gM}(\F_t(V))$, the concept algebra of $\gM$. 
Let $\overline{n}=\{0,1,...,n-1\}$ and let $T$ denote the set of functions $\tau:\overline{n}\to\overline{n}$. For $\tau\in T$ let 
$R_i(\tau v)$ denote the formula $R_i(v_{\tau(0)}\dots v_{\tau(n-1})$.
Then $\F_t(V)$ is freely generated by $\{ R_i(\tau v) : \tau\in T, i=1,2,3\}$.
We define $h:\F_t(V)\to\gA$ as follows. For $\{ i,j,k\}=\{ 1,2,3\}$, 
\begin{description}
	\item{} $h(R_i(\tau v))=\{ \tau\circ s : s\in R_i^{\gM}\}$\quad if $\tau^{-1}(0)<\tau^{-1}(1)$ and $\tau$ is one-one,
	\item{} $h(R_i(\tau v))=\{ \tau\circ s : s\in R_j^{\gM}\cup R_k^{\gM}\}$\quad if $\tau^{-1}(0)>\tau^{-1}(1)$ and $\tau$ is one-one,
	\item{} $h(R_i(\tau v))=\emptyset$\quad otherwise.
\end{description}

The proof for $\sim_{\cL}\subseteq\ker(h)$ failing is as before, with the obvious modifications. Let $t=t_1\cup t_2\cup t_3$ where $t_i$ is the similarity type with one $n$-place relation symbol $R_i$ and let $\sim_i$ be the tautological congruence relation of $\cL_i=\FOL_{t_i}(V)$, for $i=1,2,3$.

To show  $\sim_1\subseteq\ker(h)$, take the subalgebra  $\gB$ of $\gA$ generated by $\{ h(R_1(\tau v)) : \tau\in T\}$. 
Let $E$ be the set of all functions $\varepsilon:\{\langle i,j\rangle : i,j<n\}\to\{ 0,1\}$. For $\sigma\in T$ and $\varepsilon\in E$ let
\[ U(\sigma,\varepsilon) = \{ s\in {}^nM : (\forall i,j<n)( s_i\in U_{\sigma(i)}\mbox{ and }[s_i=s_j\mbox{ iff }\varepsilon(i,j)=0] )\}.\]
It is not hard to see that the atoms of $\gB$ are the nonempty elements of 
\[ \{ h(R_1(\tau v)) : \tau\in T\} \cup \{ U(\sigma,\varepsilon) : \sigma\in T, \varepsilon\in E\}.\] 
Let us define the model $\gN=\langle M,R_1^{\gN}\rangle$ as
\[ R_1^{\gN} = [\{ s\in W : s_0<s_1\}\setminus\{\langle 0,n+6\rangle\}]\cup\{\langle n+6,0\rangle\} .\]
It is not hard to see that $\{ \langle h(\phi),\mng_{\gN}(\phi)\rangle : \phi\in F_{t_1}(V)\}$ is an isomorphism between $\gB$ and the concept algebra of $\gN$, and we are done as in the case of $|V|=2$.
\end{proof}

We note that $\LL_{FOL}(V)$ is a substitutional logic family for $|V|=1$, see Remark \ref{remark:FOLuresegyelemu}.

\bigskip

\section*{Acknowledgements}
Zal\'an Gyenis was supported by the grant 2019/34/E/HS1/00044 of the National Science Centre (Poland), and by the grant of the Hungarian National Research, Development and Innovation Office, contract number: K-134275.

\end{document}